\documentclass[]{article}
\usepackage{graphicx}
\usepackage{amsfonts}
\usepackage{amsmath}
\usepackage{amssymb}
\usepackage{fancyhdr}
\usepackage{titlesec}
\usepackage{indentfirst}
\usepackage{booktabs}
\usepackage{verbatim}
\usepackage{color}
\usepackage{amsthm}
\usepackage{subfigure}

\usepackage[page,header]{appendix}
\usepackage{titletoc}

\newcommand{\rd}{\,\mathrm{d}}

\numberwithin{equation}{section}
\newtheorem{theorem}{Theorem}[section]
\newtheorem{lemma}[theorem]{Lemma}

\newtheorem{proposition}[theorem]{Proposition}
\newtheorem{definition}[theorem]{Definition}
\newtheorem{remark}[theorem]{Remark}

\def\eps{\varepsilon}

\begin{document}

\title{Uniform accuracy of implicit-explicit Runge-Kutta (IMEX-RK) schemes for hyperbolic systems with relaxation}

\author{Jingwei Hu\footnote{Department of Applied Mathematics, University of Washington, Seattle, WA 98195 (hujw@uw.edu).} \
	\ and \  Ruiwen Shu\footnote{Department of Mathematics, University of Georgia, Athens, GA 30602 (ruiwen.shu@uga.edu).} }    
	\maketitle
	
\begin{abstract}
Implicit-explicit Runge-Kutta (IMEX-RK) schemes are popular methods to treat multiscale equations that contain a stiff part and a non-stiff part, where the stiff part is characterized by a small parameter $\varepsilon$. In this work, we prove rigorously the uniform stability and uniform accuracy of a class of IMEX-RK schemes for a linear hyperbolic system with stiff relaxation. The result we obtain is optimal in the sense that it holds regardless of the value of $\varepsilon$ and the order of accuracy is the same as the design order of the original scheme, i.e., there is no order reduction.
\end{abstract}

{\small 
{\bf Key words.} Hyperbolic relaxation system, stiff, implicit-explicit Runge-Kutta method, energy estimate.

{\bf AMS subject classifications.} 35L03, 65L04, 65L06, 65M12.
}

\section{Introduction}

Many hyperbolic and kinetic equations exhibit small temporal and spatial scales, leading to different asymptotic limits \cite{Jin22}. One simple example is the following linear hyperbolic system with relaxation \cite{CLL94}:
\begin{equation}\label{eq2}\left\{\begin{split}
& \partial_t u + \partial_x v = 0, \\
& \partial_t v + \partial_x u = \frac{1}{\eps}(bu-v),
\end{split}\right.\end{equation}
where $u=u(t,x)$, $v=v(t,x)$ are unknown functions of time $t$ and position $x$, $b$ is a constant satisfying $|b|<1$, and $\varepsilon>0$ is the relaxation parameter. The value of $\varepsilon$ can range from $\varepsilon =O(1)$ (non-stiff regime) to $\varepsilon \ll 1$ (stiff regime). In particular, when $\varepsilon\rightarrow 0$, it is easy to see that the formal asymptotic limit of (\ref{eq2}) is a convection equation:
\begin{equation}
\partial_t u +\partial_x(bu)=0;
\end{equation}
if keeping the $O(\varepsilon)$ term, one can obtain a convection-diffusion equation:
\begin{equation}
\partial_t u +\partial_x(bu)=\varepsilon(1-b^2)\partial_{xx}u.
\end{equation}

Due to its multiscale nature, a popular time discretization method for system (\ref{eq2}) is the implicit-explicit (IMEX) schemes, including the IMEX Runge-Kutta (RK) methods (e.g., \cite{ARS97, KC03, PR05, DP13}) and IMEX multistep methods (e.g., \cite{ARW95, DP17, ADP20}). In these schemes, the non-stiff convection term is treated explicitly and the possibly stiff relaxation term is treated implicitly. As such, the schemes are expected to be stable under the CFL condition coming only from the convection part, hence are efficient regardless of the value of $\varepsilon$. Furthermore, the accuracy of the IMEX schemes is guaranteed (i.e., there is no order reduction) when $\varepsilon=O(1)$ by their design and when $\varepsilon\rightarrow 0$ if they are asymptotic-preserving \cite{Jin99}. However, how these schemes behave in the intermediate regime ($0<\varepsilon<O(1)$) is a difficult problem. Most studies are based on asymptotic analysis hence only address the issue to a certain extent \cite{Boscarino07, HZ17} (the error estimates typically depend on both time step $\Delta t$ and $\varepsilon$).

In our previous work \cite{HS21}, we proved rigorously the {\it uniform accuracy} of IMEX-BDF schemes, a class of IMEX multistep methods, applied to system (\ref{eq2}). Our result can be simply summarized as follows:
\begin{equation} \label{error}
\|u(T,\cdot)-U(T,\cdot)\|_{L^2}+\|v(T,\cdot)-V(T,\cdot)\|_{L^2} \leq C \Delta t^q,
\end{equation}
where $U$ and $V$ are the numerical solutions at time $T$, $q$ is the order of the scheme, and $C$ is a constant depending on $T$, $q$, etc., but {\it independent of $\eps$}. Note that this is an optimal error bound that holds for any values of $\eps$, so the accuracy of the scheme is guaranteed in all regimes!

The goal of this work is to establish a similar result for IMEX-RK schemes. This turns out to be much more delicate. Based on the previous study \cite{Boscarino07}, it is known that most of the popular IMEX-RK schemes will suffer from the order reduction in the intermediate regime. For example, the widely used ARS(4,4,3) scheme \cite{ARS97}, which is a third order method by design, will reduce to second order when $\varepsilon$ is somewhere between 0 and 1 (in the sense of (\ref{error})). Therefore, an optimal uniform accuracy result as in (\ref{error}) cannot be expected for general IMEX-RK schemes. Nevertheless, there are some exceptions: 1) It has been numerically observed that the second order ARS(2,2,2) scheme can maintain the uniform second-order accuracy for a wide range of $\varepsilon$; 2) In \cite{Boscarino09, BR09}, a third order IMEX-RK scheme, BHR(5,5,3), is constructed by imposing additional order conditions and the uniform third-order accuracy is observed in several test problems. The structure of these two special IMEX-RK schemes exactly motivates our current study. Using the energy estimates, we are able to prove that, for the prototype problem (\ref{eq2}), a class of second and third order IMEX-RK schemes can maintain their order of accuracy regardless of the value of $\varepsilon$. As a byproduct, we also answer rigorously why certain IMEX-RK schemes, such as ARS(4,4,3), exhibit reduced uniform second order accuracy.

The rest of this paper is organized as follows. In Section~\ref{sec:pre}, we recall the regularity of the solution to (\ref{eq2}), its IMEX-RK time discretization as well as spatial discretization. We establish the uniform stability of a class of IMEX-RK schemes in Section~\ref{sec:uniform_stability}. We then prove, respectively, in Section~\ref{sec_second} and Section~\ref{sec_third}, the uniform second order accuracy and uniform third order accuracy of the IMEX-RK schemes. Numerical tests are presented in Section~\ref{sec:num} to validate the theoretical results obtained in this paper.

\section{Preliminaries}
\label{sec:pre}

In this section we present some basic settings of this paper, including the regularity of the solution, the time discretization, and the spatial discretization. We always assume $x \in [0, 2\pi]$ with periodic boundary condition. All integrals without range refer to $\int_0^{2\pi} \cdot \,\rd{x}$ and all norms $\|\cdot \|$ without subscript refer to the $L^2$ norm in $x$. $c$ and $C$ denote small/large positive constants independent of $\varepsilon$ ($c$ and $C$ may change from line to line).

Our starting point to treat (\ref{eq2}) is to introduce a change of variable $w=v-bu$ and formulate the system equivalently as
\begin{equation}\label{eq1}\left\{\begin{split}
& \partial_t u + \partial_x (bu+w) = 0, \\
& \partial_t w + \partial_x ((1-b^2)u-bw) = -\frac{1}{\eps}w.
\end{split}\right.\end{equation}

In \cite[Theorem 3.1]{HS21} we proved the following regularity result of the solution to \eqref{eq1}, which we cite without proof.
\begin{lemma}\label{lem_reg}
For any integer $s\ge 0$, assume 
\begin{equation}\label{thm_reg_0}
\|u(0)\|_{H^s}^2 + \|w(0)\|_{H^s}^2 =: E_0 < \infty.
\end{equation}
Then for all $t\ge s\eps \log(1/\eps)$, the solution to \eqref{eq1} satisfies
\begin{equation} \label{energy1}
\|u(t)\|_{H^s}^2 + \|w(t)\|_{H^s}^2 \le CE_0,
\end{equation}
also
\begin{equation}\label{thm_reg_1}
\|\partial_t^{r_1}\partial_x^{r_2} u(t)\|^2 + \|\partial_t^{r_1}\partial_x^{r_2} w(t)\|^2 \le CE_0,\quad r_1+r_2 \le s,
\end{equation}
and
\begin{equation}\label{thm_reg_2}
\|\partial_t^{r_1}\partial_x^{r_2} w(t)\|^2 \le CE_0\eps^2, \quad r_1+r_2 \le s-1.
\end{equation}
\end{lemma}

\subsection{IMEX-RK schemes for time discretization}
\label{subsec:IMEX}

A general $s$-stage IMEX-RK scheme applied to the system (\ref{eq1}) consists of an explicit treatment to the non-stiff convection term and an implicit one to the stiff relaxation term \cite{PR05}:
\begin{equation}\label{sch}\begin{split}
& U^{(i)} = U^n -\Delta t \sum_{j=1}^{i-1} \tilde{a}_{i j} (b\partial_x U^{(j)} + \partial_x W^{(j)}),  \quad i=1,\dots,s,\\
& W^{(i)} = W^n -\Delta t \sum_{j=1}^{i-1}\tilde{a}_{i j} ((1-b^2)\partial_x U^{(j)} - b\partial_x W^{(j)}) - \frac{\Delta t}{\varepsilon}\sum_{j=1}^i a_{i j}W^{(j)}, \quad i=1,\dots,s,\\
& U^{n+1} = U^n -\Delta t \sum_{j=1}^{s} \tilde{b}_j (b\partial_x U^{(j)} + \partial_x W^{(j)}), \\
& W^{n+1} = W^n -\Delta t \sum_{j=1}^{s}\tilde{b}_j ((1-b^2)\partial_x U^{(j)} - b\partial_x W^{(j)}) - \frac{\Delta t}{\varepsilon}\sum_{j=1}^s b_j W^{(j)}, 
\end{split}\end{equation}
where the matrix $\tilde{A}=(\tilde{a}_{ij})\in \mathbb{R}^{s\times s}$ is strictly lower-triangular ($\tilde{a}_{ij}=0$ for $j\geq i$) and the matrix $A=(a_{ij})\in \mathbb{R}^{s\times s}$ is lower-triangular ($a_{ij}=0$ for $j>i$). Along with the vectors $\vec{\tilde{b}}=(\tilde{b}_1,\dots,\tilde{b}_s)$ and $\vec{b}=(b_1,\dots,b_s)$ 
they can be represented by a double Butcher tableau:
\begin{equation} \label{tableau}
\centering
\begin{tabular}{c|c}
$\vec{\tilde{c}}$ & $\tilde{A}$\\
\hline
& $\vec{\tilde{b}}$
\end{tabular} \quad \quad
\begin{tabular}{c|c}
$\vec{c}$ & $A$\\
\hline
& $\vec{b}$
\end{tabular}
\end{equation}
with the vectors $\vec{\tilde{c}}=(\tilde{c}_1,\dots,\tilde{c}_s)^\top$ and $\vec{c}=(c_1,\dots,c_s)^\top$ defined by
\begin{equation} 
\tilde{c}_i=\sum_{j=1}^{i-1}\tilde{a}_{ij}, \quad c_i=\sum_{j=1}^{i}a_{ij}.
\end{equation}
The tableau (\ref{tableau}) must satisfy the standard order conditions \cite{PR05}. According to the structure of matrix $A$ in the implicit tableau, the IMEX-RK schemes can be classified into several categories \cite{BPR13, DP13}. In this paper, we restrict our study to the IMEX-RK schemes of type CK \cite{KC03} and with implicitly-stiffly-accurate (ISA) property:
\begin{itemize}
\item {\bf Type CK}: if the matrix $A$ can be written as
\begin{equation}  \label{CK}
\left(
 \begin{matrix}
  0 & 0 \\
  \vec{a} & \hat{A}
  \end{matrix}
\right),
\end{equation}
where the vector $\vec{a}\in \mathbb{R}^{s-1}$ and the submatrix $\hat{A}\in \mathbb{R}^{(s-1)\times (s-1)}$ is invertible; in particular, if $\vec{a}=\vec{0}$, $b_1=0$, the scheme is of type ARS \cite{ARS97}.
\item If $a_{s i}=b_i$, $i=1,\dots,s$, the scheme is said to be {\bf implicitly stiffly accurate (ISA)}\footnote{If, in addition, $\tilde{a}_{si}=\tilde{b}_i$, $i=1,\dots,s$, the scheme is said to be {\bf globally stiffly accurate (GSA)}.}.
\end{itemize}

Therefore, if the scheme is of type CK, we can see that $U^{(1)}=U^n$ and $W^{(1)}=W^n$. Define the vectors
\begin{equation}\label{vecUW}
\vec{U} = (U^{(1)},\dots,U^{(s)})^\top,\quad \vec{W} = (W^{(1)},\dots,W^{(s)})^\top,\quad \vec{e}=(1,\dots,1)^\top,
\end{equation}
we may rewrite the first $s$ stages in \eqref{sch} as
\begin{equation}\label{sch_vec}\begin{split}
& \vec{U} = U^{n}\vec{e} -\Delta t \tilde{A} (b\partial_x \vec{U} + \partial_x \vec{W}),  \\
& \vec{W} = W^{n}\vec{e} -\Delta t \tilde{A} ((1-b^2)\partial_x \vec{U} - b\partial_x \vec{W})  - \frac{\Delta t}{\varepsilon} A \vec{W}.
\end{split}\end{equation}

\subsection{Spatial discretization}

For spatial discretization, we adopt the same Fourier-Galerkin spectral method as in \cite{HS21}. Consider the space of trigonometric polynomials of degree up to $N$:
\begin{equation}
\mathbb{P}_N = \text{span}\{e^{ikx}:-N\le k \le N\}
\end{equation}
equipped with inner product
\begin{equation}
\langle f,g \rangle = \frac{1}{2\pi}\int f \bar{g}\rd{x}.
\end{equation}
For a function $f\in L^2[0,2\pi]$, denote $\mathcal{P}_N f$ as its orthogonal projection onto $\mathbb{P}_N$. We have the following basic facts:
\begin{lemma} \label{Fourier1}
For any $2\pi$-periodic function $f(x)\in H^s[0,2\pi]$, there holds
\begin{equation}
\|(I-\mathcal{P}_N)f\|\leq \frac{1}{N^s}\|f\|_{H^s}.
\end{equation}
\end{lemma}
\begin{lemma} \label{Fourier2}
For any function $\phi(x)\in \mathbb{P}_N$, there holds
\begin{equation}
\| \phi^{(s)}\|\leq N^{s}\|\phi\|.
\end{equation}
In particular, if $\Delta t \le c_{CFL}/N^2$, where $c_{CFL}>0$ is some constant, one has
\begin{equation}\label{CFL1}
\Delta t\| \partial_x\phi \|^2 \le c_{CFL} \| \phi \|^2,\quad \forall \, \phi \in \mathbb{P}_N.
\end{equation}
\end{lemma}

The Fourier-Galerkin spectral method for (\ref{sch}) seeks to approximate $U^{(i)}$, $W^{(i)}$ as
\begin{equation} \label{FG}
U^{(i)}(x) \approx \sum_{k=-N}^N U^{(i)}_k e^{ikx} =: U^{(i)}_N(x), \quad W^{(i)}(x) \approx \sum_{k=-N}^N W^{(i)}_k e^{ikx} =: W^{(i)}_N(x),
\end{equation}
similarly for $U^{n+1}$ and $W^{n+1}$. Substituting (\ref{FG}) into \eqref{sch} and conducting the Galerkin projection yields the fully discrete scheme
\begin{equation}\label{sch_full}\begin{split}
& U^{(i)}_N = U^n_N -\Delta t \sum_{j=1}^{i-1} \tilde{a}_{i j} (b\partial_x U^{(j)}_N + \partial_x W_N^{(j)}), \\
& W^{(i)}_N = W^n_N -\Delta t \sum_{j=1}^{i-1}\tilde{a}_{i j} ((1-b^2)\partial_x U^{(j)}_N - b\partial_x W^{(j)}_N) - \frac{\Delta t}{\varepsilon}\sum_{j=1}^i a_{i j}W^{(j)}_N,\\
& U^{n+1}_N = U^n_N -\Delta t \sum_{j=1}^{s} \tilde{b}_j (b\partial_x U^{(j)}_N + \partial_x W^{(j)}_N), \\
& W^{n+1}_N = W^n_N -\Delta t \sum_{j=1}^{s}\tilde{b}_j ((1-b^2)\partial_x U^{(j)}_N - b\partial_x W^{(j)}_N) - \frac{\Delta t}{\varepsilon}\sum_{j=1}^s b_j W^{(j)}_N.
\end{split}\end{equation}
The initial condition is given by
\begin{equation}\label{UN0}
U_N^0=\mathcal{P}_N u_{in}(x), \quad  W_N^0=\mathcal{P}_N w_{in}(x),
\end{equation}
where $(u_{in},w_{in})$ denotes the initial condition of (\ref{eq1}).

Note that due to the linearity of (\ref{eq1}), the semi-discrete scheme (\ref{sch}) formally looks the same as the fully discrete scheme (\ref{sch_full}), except for the initial condition. In the following, we will refer to the notation in Section~\ref{subsec:IMEX} and neglect the subscript $N$ whenever it does not cause confusion in the context.

\section{Uniform stability}
\label{sec:uniform_stability}


As an attempt to prove the uniform stability of \eqref{sch} by energy estimates, we take an $s\times s$ matrix $M$ to be determined and aim to left multiply it to \eqref{sch_vec}. We may choose the first column of $M$ as zero without loss of generality since the first row of \eqref{sch_vec} is trivial.

Multiplying the $\vec{W}$ equation in (\ref{sch_vec}) by $\vec{W}^\top M$ from the left, we get a scalar equation
\begin{equation}
\vec{W}^\top M \vec{W} = W^{n}\vec{W}^\top M\vec{e} -\Delta t \vec{W}^\top M \tilde{A} ((1-b^2)\partial_x \vec{U} - b\partial_x \vec{W})  - \frac{\Delta t}{\varepsilon} \vec{W}^\top M A \vec{W},
\end{equation}
i.e.,
\begin{equation}\label{vec_energy1}
(W^{(s)})^2 = (W^{n})^2 -  \underline{\vec{W}^\top M_* \vec{W}}   - \underline{\frac{\Delta t}{\varepsilon} \vec{W}^\top M A \vec{W}} -\Delta t \vec{W}^\top M \tilde{A} ((1-b^2)\partial_x \vec{U} - b\partial_x \vec{W}),
\end{equation}
where
\begin{equation}
M_* := M \begin{pmatrix}
0 & 0 & 0 & \cdots & 0 \\
-1 & 1 & 0 & \cdots & 0 \\
-1 & 0 & 1 & \cdots & 0 \\
\cdots & \cdots & \cdots & \cdots & \cdots \\
-1 & 0 & 0 & \cdots & 1 \\
\end{pmatrix} + \begin{pmatrix}
1 & 0 & 0 & \cdots & 0 \\
0 & 0 & 0 & \cdots & 0 \\
0 & 0 & 0 & \cdots & 0 \\
\cdots & \cdots & \cdots & \cdots & \cdots \\
0 & 0 & 0 & \cdots & -1 \\
\end{pmatrix}.
\end{equation}
If the two underlined terms are semi-positive-definite quadratic forms in $\vec{W}$, then we can gain good terms in an energy estimate for $\vec{W}$. This motivates the following conditions.

For the matrix $A$ in a type CK IMEX-RK scheme, we assume there exists a matrix $M$ such that
\begin{itemize}
\item[{\bf (M1)}] $M A+(MA)^\top$ is semi-positive-definite and has rank $s-1$.
\item[{\bf (M2)}] $M_*+M_*^\top$ is semi-positive-definite and has rank $s-1$.
\end{itemize}

We first state a lemma on semi-positive-definite matrices.
\begin{lemma}\label{lem_P}
Let $P$ be an $s\times s$ symmetric semi-positive-definite matrix with rank $s-1$. Assume $\vec{u}\in\mathbb{R}^s$ satisfies $\vec{u}^\top P \vec{u}=0$. Then $P\vec{u}=\vec{0}$, and the quadratic form $\vec{x}^\top P \vec{x}$ is positive-definite when restricted to the orthogonal complement of $\vec{u}$.
\end{lemma}

\begin{proof}
We first show that $P\vec{u}=\vec{0}$. In fact, denoting $P \vec{u}=\vec{\xi}$, we have
\begin{equation}
0\le (\vec{u}-\alpha\vec{\xi})^\top P (\vec{u}-\alpha\vec{\xi}) = -2\alpha |\vec{\xi}|^2+ \alpha^2 \vec{\xi}^\top P\vec{\xi}
\end{equation}
for any $\alpha>0$. Sending $\alpha\rightarrow 0$, we get $\vec{\xi}=\vec{0}$.

We may assume $|\vec{u}|=1$ without loss of generality. Since $\vec{u}$ is an eigenvector of the symmetric matrix $P$ with eigenvalue 0, we may extend it to an orthonormal basis $\vec{u},\vec{u}_2,\dots,\vec{u}_s$ of $\mathbb{R}^s$ consisting of eigenvectors of $P$. Denoting $U=(\vec{u},\vec{u}_2,\dots,\vec{u}_s)$, we then have $P=UDU^\top$ where $D=\text{diag}\{0,\lambda_2,\dots,\lambda_s\}$ is a diagonal matrix with diagonal entries being eigenvalues of $P$. Since $P$ is semi-positive-definite matrix with rank $s-1$, the eigenvalues $\lambda_2,\dots,\lambda_s$ are positive. Therefore we see that
\begin{equation}
\vec{x}^\top P \vec{x} = \vec{x}^\top UDU^\top \vec{x} = \sum_{j=2}^s \lambda_j (\vec{u}_j\cdot \vec{x})^2.
\end{equation}
Therefore this quadratic form is positive-definite when restricted to $\text{Span}\{\vec{u}_2,\dots,\vec{u}_s\}$, i.e., the orthogonal complement of $\vec{u}$.
\end{proof}

We then have the following lemma:
\begin{lemma}\label{lem_v1}
Assume {\bf (M2)}, then there exists a constant $c>0$ such that
\begin{equation}
\vec{\xi}^\top M_* \vec{\xi} \ge c\sum_{1\le i < j \le s} |\xi_i-\xi_j|^2,
\end{equation}
for any vector $\vec{\xi}\in\mathbb{R}^s$.
\end{lemma}
\begin{proof}

The condition {\bf (M2)}, combined with the easily verified identity $\vec{e}^\top (M_* + M_*^\top) \vec{e}=0$, allows us to apply Lemma \ref{lem_P} with $P=M_* + M_*^\top$. This gives that  $(M_* + M_*^\top) \vec{e}=0$, and the matrix $M_* + M_*^\top$ is strictly positive-definite when restricted to the $\vec{e}^\perp$, the orthogonal complement of $\vec{e}$. Notice that each vector of the form $\vec{e}_{ij}:=(0,\dots,0,1,0,\dots,0,-1,0,\dots,0)$ with $i$-th entry 1 and $j$-th entry $-1$ lies in $\vec{e}^\perp$. Therefore, denoting $\vec{\eta}$ as the orthogonal projection of $\vec{\xi}$ onto $\vec{e}^\perp$, we have $\vec{\xi}^\top M_* \vec{\xi} = \vec{\eta}^\top M_* \vec{\eta} \ge c|\vec{\eta}|^2 \ge c|\vec{\eta}\cdot\vec{e}_{ij}|^2 = c|\vec{\xi}\cdot\vec{e}_{ij}|^2 =  c|\xi_i-\xi_j|^2$ for each $1\le i < j \le s$. Summing over $i,j$ and dividing by $s(s-1)/2$, we get the conclusion.
\end{proof}

A few IMEX-RK schemes, ARS(2,2,2), ARS(4,4,3), and BHR(5,5,3)* and their corresponding $M$ matrices are provided in Appendix~\ref{app:IMEX-RK}. We also discuss in Appendix~\ref{app:M} necessary conditions to find $M$.
\begin{remark}
Our method of the multiplier matrix $M$ is inspired by the stability analysis in \cite[Section 3.1]{FS16}. For the ARS type schemes such as ARS(2,2,2) and ARS(4,4,3), one can translate the corresponding energy estimates into our formulation and obtain the $M$ matrices, see Appendix \ref{app:IMEX-RK}. However, our method can handle more general IMEX-RK schemes of type CK: one example is the BHR(5,5,3)* scheme.
\end{remark}


We are ready to state our main result of this section.

\begin{theorem}[Uniform stability of IMEX-RK schemes]\label{thm_stab}
Consider the fully discrete scheme (\ref{sch_full})-(\ref{UN0}) for system (\ref{eq1}). Assume the time discretization is the IMEX-RK scheme of type CK and ISA, for which there exists a matrix $M$ satisfying {\bf (M1)} and {\bf (M2)}. Let $c_{CFL}$ be any fixed positive number. Then for any time $T>0$ and $n\in\mathbb{Z}_{\ge 0}$ with $n\Delta t\le T$, we have
\begin{equation}
(1-b^2)\|U^n_N\|^2+\|W^n_N\|^2 \le C((1-b^2)\|u_{in}\|^2+\|w_{in}\|^2),
\end{equation}
under the condition $\Delta t \leq \min(c_{CFL}/N^2,c)$. Here $c$ and $C$ are positive constants independent of $\varepsilon$, $N$ and $\Delta t$.
\end{theorem}

The IMEX-RK schemes ARS(2,2,2), ARS(4,4,3), and BHR(5,5,3)* given in the Appendix all satisfy the assumptions in Theorem~\ref{thm_stab}, hence they are uniformly stable.

\begin{proof}
We start from \eqref{vec_energy1}. Combined with a similar equation for $U$ and integrated in $x$, we obtain
\begin{equation}\label{vec_energy2}\begin{split}
& (1-b^2)\|U^{(s)}\|^2+\|W^{(s)}\|^2 \\
= & (1-b^2)\|U^{n}\|^2+\|W^{n}\|^2 -  (1-b^2)\int\vec{U}^\top M_* \vec{U}\rd{x} - \int\vec{W}^\top M_* \vec{W}\rd{x} - \frac{\Delta t}{\varepsilon} \int \vec{W}^\top M A \vec{W}\rd{x}\\
& -\Delta t b(1-b^2)\int \vec{U}^\top M \tilde{A} \partial_x \vec{U}\rd{x} +\Delta t b\int \vec{W}^\top M \tilde{A} \partial_x \vec{W}\rd{x} \\
& - \Delta t (1-b^2)\sum_{i,j=1}^s (M \tilde{A})_{i j} \int (U^{(i)}\partial_x W^{(j)}+W^{(i)}\partial_x U^{(j)} )\rd{x}.
\end{split}\end{equation}
Using {\bf (M1)} and Lemma~\ref{lem_v1} we deduce that
\begin{equation}\label{vec_energy3}\begin{split}
 (1-b^2)\|U^{(s)}\|^2+\|W^{(s)}\|^2 
\le & (1-b^2)\|U^{n}\|^2+\|W^{n}\|^2 -  (1-b^2)c \|\delta \vec{U}\|^2 - c  \|\delta \vec{W}\|^2 \\
& -\underline{\Delta t b(1-b^2)\int \vec{U}^\top M \tilde{A} \partial_x \vec{U}\rd{x}} +\underline{\Delta t b\int \vec{W}^\top M \tilde{A} \partial_x \vec{W}\rd{x}} \\
& - \Delta t (1-b^2)\sum_{i,j=1}^s (M \tilde{A})_{i j} \int (U^{(i)}\partial_x W^{(j)}+W^{(i)}\partial_x U^{(j)})\rd{x},
\end{split}\end{equation}
where we denote
\begin{equation}
\|\delta \vec{U}\|^2 := \sum_{1\le i < j \le s} \|U^{(i)}-U^{(j)}\|^2,
\end{equation}
and similarly for $W$.

In the two underlined terms, each element is of the form $C\Delta t \int U^{(i)}\partial_x U^{(j)}\rd{x}$ or its counterpart in $W$. It can be estimated by 
\begin{equation}\label{Uij_est}\begin{split}
\left|\Delta t\int U^{(i)}\partial_x U^{(j)}\rd{x}\right| = & \left|\Delta t\int (U^{(i)}-U^{(j)})\partial_x U^{(j)}\rd{x}\right| \le c\|U^{(i)}-U^{(j)}\|^2 + C\Delta t^2\|\partial_x U^{(j)}\|^2 \\
 \le & c\|U^{(i)}-U^{(j)}\|^2 + C\Delta t\|U^{(j)}\|^2,
\end{split}\end{equation}
where the new term added in the first equality is zero due to integration by parts and periodic boundary condition; the first inequality is by Young's inequality; and the last inequality uses the property $\Delta t\|\partial_x U^{(j)}\|^2 \le c_{CFL} \|U^{(j)}\|^2$ by \eqref{CFL1}. Here $c\|U^{(i)}-U^{(j)}\|^2$, with the pre-factor $c$ chosen sufficiently small, can be absorbed by the good term $(1-b^2)c \|\delta \vec{U}\|^2$. The term  $C\Delta t\|U^{(j)}\|^2$ can be estimated by 
\begin{equation}\label{Uj_est}
C\Delta t\|U^{(j)}\|^2 \le 2C\Delta t(\|U^{(1)}\|^2 + \|U^{(j)}-U^{(1)}\|^2) =  2C\Delta t(\|U^{n}\|^2 + \|U^{(j)}-U^{(1)}\|^2),
\end{equation}
and $2C\Delta t\|U^{(j)}-U^{(1)}\|^2$ can be absorbed by the good term $(1-b^2)c \|\delta \vec{U}\|^2$ provided $\Delta t\le c$ as assumed. 

The last term in \eqref{vec_energy3} can be treated similarly as \eqref{Uij_est} using the identity
\begin{equation}\label{cross}\begin{split}
\int (U^{(i)}\partial_x W^{(j)}+W^{(i)}\partial_x U^{(j)})\rd{x} = \int (U^{(i)}-U^{(j)})\partial_x W^{(j)}\rd{x} - \int \partial_x U^{(j)}( W^{(j)}- W^{(i)})\rd{x},
\end{split}\end{equation}
where integration by parts and periodic boundary is used.

Therefore we conclude with
\begin{equation}\begin{split}
& (1-b^2)\|U^{(s)}\|^2+\|W^{(s)}\|^2 
\le  (1+C\Delta t)((1-b^2)\|U^{n}\|^2+\|W^{n}\|^2).
\end{split}\end{equation}
If the IMEX-RK scheme is GSA (such as ARS(2,2,2) and ARS(4,4,3)), then $U^{(s)}=U^{n+1}$ and $W^{(s)}=W^{n+1}$ and we can jump directly to step (\ref{secondlast}) and finish the proof. For a general ISA scheme (such as BHR(5,5,3)*) we still need to estimate the change from stage $s$ to step $n+1$.

To obtain the same estimate for $U^{n+1},W^{n+1}$, we write the last two equations of \eqref{sch} as
\begin{equation}\label{energyn0}\begin{split}
& U^{n+1} = U^{(s)} -\Delta t \sum_{j=1}^{s} (\tilde{b}_j-\tilde{a}_{s j}) (b\partial_x U^{(j)} + \partial_x W^{(j)}), \\
& W^{n+1} = W^{(s)} -\Delta t \sum_{j=1}^{s}(\tilde{b}_j-\tilde{a}_{s j}) ((1-b^2)\partial_x U^{(j)} - b\partial_x W^{(j)}),
\end{split}\end{equation}
where the ISA property guarantees that no stiff terms appear here. Multiplying (\ref{energyn0}) by $2U^{n+1},2W^{n+1}$ respectively and integrating in $x$ gives the energy estimate
\begin{equation}\label{energyn1}\begin{split}
& (1-b^2)\|U^{n+1}\|^2+\|W^{n+1}\|^2 \\
= & (1-b^2)\|U^{(s)}\|^2+\|W^{(s)}\|^2 -  (1-b^2)\|U^{n+1}-U^{(s)}\|^2 -   \|W^{n+1}-W^{(s)}\|^2  \\
& -2\Delta t b(1-b^2)\sum_{j=1}^{s} (\tilde{b}_j-\tilde{a}_{s j}) \int U^{n+1}\partial_x U^{(j)}\rd{x}  +2\Delta t b\sum_{j=1}^{s} (\tilde{b}_j-\tilde{a}_{s j}) \int W^{n+1}\partial_x W^{(j)}\rd{x} \\
& - 2\Delta t (1-b^2)\sum_{j=1}^{s} (\tilde{b}_j-\tilde{a}_{s j}) \int (U^{n+1}\partial_x W^{(j)} +W^{n+1}\partial_x U^{(j)} )\rd{x}.
\end{split}\end{equation}
Adding with \eqref{vec_energy3} and treating the bad terms in \eqref{vec_energy3} as before, we obtain
\begin{equation}\begin{split}
& (1-b^2)\|U^{n+1}\|^2+\|W^{n+1}\|^2 \\
\le & (1+C\Delta t)((1-b^2)\|U^{n}\|^2+\|W^{n}\|^2) -  c \|\delta \vec{U}\|^2 - c  \|\delta \vec{W}\|^2 -   (1-b^2)\|U^{n+1}-U^{(s)}\|^2 -   \|W^{n+1}-W^{(s)}\|^2  \\
& -2\Delta t b(1-b^2)\sum_{j=1}^{s} (\tilde{b}_j-\tilde{a}_{s j}) \int U^{n+1}\partial_x U^{(j)}\rd{x}  +2\Delta t b\sum_{j=1}^{s} (\tilde{b}_j-\tilde{a}_{s j}) \int W^{n+1}\partial_x W^{(j)}\rd{x} \\
& - 2\Delta t (1-b^2)\sum_{j=1}^{s} (\tilde{b}_j-\tilde{a}_{s j}) \int (U^{n+1}\partial_x W^{(j)} +W^{n+1}\partial_x U^{(j)} )\rd{x}.
\end{split}\end{equation}
Then, similarly to \eqref{Uij_est}, we have 
\begin{equation}
\Delta t\left|\int U^{n+1}\partial_x U^{(j)}\rd{x}\right| \le c\|U^{n+1}-U^{(j)}\|^2 + C\Delta t\|U^{(j)}\|^2.
\end{equation}
Notice that $\|U^{n+1}-U^{(j)}\|^2\le 2(\|U^{n+1}-U^{(s)}\|^2 + \|U^{(s)}-U^{(j)}\|^2) \le 2(\|U^{n+1}-U^{(s)}\|^2 + \|\delta \vec{U}\|^2)$ can be controlled by the good terms $c \|\delta \vec{U}\|^2$ and $ (1-b^2) \|U^{n+1}-U^{(s)}\|^2$; and $C\Delta t\|U^{(j)}\|^2$ can be controlled by \eqref{Uj_est}. Treating the bad term $\int W^{n+1}\partial_x W^{(j)}\rd{x}$ similarly and the cross terms $\int (U^{n+1}\partial_x W^{(j)} +W^{n+1}\partial_x U^{(j)} )\rd{x}$ similarly as \eqref{cross}, we obtain
\begin{equation} \label{secondlast}
(1-b^2)\|U^{n+1}\|^2+\|W^{n+1}\|^2 
\le (1+C\Delta t)((1-b^2)\|U^{n}\|^2+\|W^{n}\|^2).
\end{equation}

The Gronwall's inequality implies (and adding back $N$)
\begin{equation}
(1-b^2)\|U_N^{n}\|^2+\|W_N^{n}\|^2 \leq \exp(CT)((1-b^2)\|U_N^{0}\|^2+\|W_N^{0}\|^2).
\end{equation}
Finally we have $\|U_N^0\|\le \|u_{in}\|$, $\|W_N^0\|\le \|w_{in}\|$ by \eqref{UN0}, hence the conclusion follows. 
\end{proof}

\section{Second order uniform accuracy}
\label{sec_second}

For the matrix $A$ in a type CK IMEX-RK scheme, we assume
\begin{itemize}
\item[{\bf (A)}] The last component of $\vec{v}$ is zero, where $\vec{v}$ is a generator of the one-dimensional null space of $A$.
\end{itemize}

We then have the following lemma:
\begin{lemma}\label{lem_v}
Assume {\bf (M1)} and {\bf (A)}, then there exists a constant $c>0$ such that
\begin{equation}
\vec{\xi}^\top MA \vec{\xi} \ge c|\xi_s|^2,
\end{equation}
for any vector $\vec{\xi}\in\mathbb{R}^s$.
\end{lemma}
\begin{proof}
Notice that $\vec{v}^\top (MA+(MA)^\top) \vec{v}=0$ by the definition of $\vec{v}$. This, together with {\bf (M1)}, allows us to apply Lemma \ref{lem_P} with $P=MA+(MA)^\top$. This gives that the matrix $MA+(MA)^\top$ is strictly positive-definite in the orthogonal complement of $\vec{v}$. {\bf (A)} implies that $(0,\dots,0,1)^\top$ is in the orthogonal complement of $\vec{v}$, and the conclusion follows by an argument similar to the proof of Lemma \ref{lem_v1}. 
\end{proof}

\begin{remark} \label{eAa}
Note that the first component of $\vec{v}$ must be nonzero and we fix it as $1$ for uniqueness. Then for type ARS schemes it is easy to see $\vec{v}=(1,0,\dots,0)^\top$ so condition {\bf (A)} is automatically satisfied. For general type CK schemes, condition {\bf (A)} is equivalent to $\vec{e}_s^\top\hat{A}^{-1}\vec{a}=0$, where $\vec{e}_s=(0,\dots,0,1)^\top\in \mathbb{R}^{s-1}$. This condition was actually required in \cite{BR09} (eqn (10)) in the construction of the scheme. So BHR(5,5,3)* satisfies {\bf (A)}.
\end{remark}

We denote the numerical error at the $n$-th time step as
\begin{equation}
U_e^n = U^n_N - u(t_n),\quad W_e^n = W^n_N - w(t_n),
\end{equation}
where $U^n_N$, $W^n_N$ are the numerical solution obtained by (\ref{sch_full})-(\ref{UN0}), and $u(t_n)$ and $w(t_n)$ are the exact solution to (\ref{eq1}) at $t_n$. We say the initial data is \emph{consistent up to order} $q$ if $\|u_{in}\|_{H^q}^2+\|w_{in}\|_{H^q}^2\le C$ and the scheme we are considering is applied after an initial layer of length $T_0\ge q \varepsilon \log(1/\varepsilon)$.

Our main result in this section is stated as follows.
\begin{theorem}[Second order uniform accuracy of IMEX-RK schemes]\label{thm_2nd}
Under the same assumptions as in Theorem \ref{thm_stab}, further assume 
\begin{itemize}
\item The IMEX-RK scheme satisfies the standard (up to) second order conditions (eqns (6)-(8) in \cite{PR05}).
\item $c_i=\tilde{c}_i,\,i=1,\dots,s$.
\item The condition {\bf (A)}.
\item The initial data is consistent up to order $6$.
\end{itemize}
Then for any $T>0$ and $n\in\mathbb{Z}_{\ge 0}$ with $n\Delta t\le T$, we have
\begin{equation} 
(1-b^2)\|U^n_e\|^2+\|W^n_e\|^2 \le  C\Big(\Delta t^4 + \frac{1}{N^8}\Big),
\end{equation}
with $C$ independent of $\varepsilon$, $N$ and $\Delta t$.
\end{theorem}

The IMEX-RK schemes ARS(2,2,2), ARS(4,4,3), and BHR(5,5,3)* given in the Appendix all satisfy the assumptions in Theorem~\ref{thm_2nd}, hence they will exhibit at least second order uniform accuracy in time.

\begin{remark}
In the study of IMEX-BDF schemes in \cite{HS21}, the uniform accuracy is basically a direct consequence of the uniform stability because the order of the scheme is the same as its stage order. However, this is not the case for IMEX-RK schemes. In fact, if the scheme is second order, the intermediate stages may only have stage order 1 (c.f., Lemma \ref{lem_LTE2}). This makes the proof of uniform accuracy significantly harder than that of uniform stability.
\end{remark}

We will prove Theorem~\ref{thm_2nd} in the rest of this section. To begin with, notice that the error from the initial projection \eqref{UN0} is bounded by
\begin{equation}
(1-b^2)\|U^0_N-u_{in}\|^2+\|W^0_N-w_{in}\|^2 \le \frac{C}{N^8},
\end{equation}
by Lemma \ref{Fourier1} and the fact that $\|u_{in}\|_{H^4}^2+\|w_{in}\|_{H^4}^2\le C$ by Lemma \ref{lem_reg}. Therefore, in the rest of the proof we may ignore the $N$-dependence. We first analyze the order of the local truncation error in Section \ref{sec_second1}. Then, in Section \ref{sec_second2} we conduct energy estimates \eqref{sch_vec_st1} and \eqref{energy_e_n01} for the error $U^n_e,W^n_e$, analogous to \eqref{vec_energy3} and \eqref{energyn1}, with extra terms coming from the local truncation error. These energy estimates directly imply the first order uniform accuracy, as shown in Section \ref{sec_second3}. We finally improve to second order uniform accuracy in Section \ref{sec_second4} with the aid of assumption {\bf (A)}.

\subsection{Local truncation error}\label{sec_second1}

With the assumption $c_i=\tilde{c}_i,\,i=1,\dots,s$, $W^{(i)}$ is supposed to be an approximation of $w(t_n+c_i\Delta t)$ (and similarly for $U^{(i)}$). With this in mind, we replace $W^{(i)}$ in \eqref{sch} by $w(t_n+c_i\Delta t)$ and define the local truncation error $E_W^{(i)}$, $i=2,\dots,s$ and $E_W^{n+1}$ by
\begin{equation}\label{sch_exact}\begin{split}
& w(t_n+c_i\Delta t) = w(t_n) -\Delta t \sum_{j=1}^{i-1}\tilde{a}_{i j} ((1-b^2)\partial_x u(t_n+c_j\Delta t) - b\partial_x w(t_n+c_j\Delta t)) - \frac{\Delta t}{\varepsilon}\sum_{j=1}^i a_{i j}w(t_n+c_j\Delta t) - E_W^{(i)},\\
& w(t_n+\Delta t) = w(t_n) -\Delta t \sum_{j=1}^{s}\tilde{b}_{j} ((1-b^2)\partial_x u(t_n+c_j\Delta t) - b\partial_x w(t_n+c_j\Delta t)) - \frac{\Delta t}{\varepsilon}\sum_{j=1}^s b_jw(t_n+c_j\Delta t) - E_W^{n+1}.
\end{split}\end{equation}
We also write $E_W^{(1)} = 0$.

Then we have the estimates for the local truncation error uniform in $\varepsilon$.
\begin{lemma}\label{lem_LTE2}
For a second order IMEX-RK scheme of type CK with $c_i=\tilde{c}_i,\,i=1,\dots,s$ and assume the initial data is consistent up to order $q\ge 3$. Then, in the $L^2$ norm,
\begin{equation}
E_W^{(i)} = O(\Delta t^2),\,i=2,\dots,s,\quad E_W^{n+1} = O(\Delta t^3),
\end{equation}
and similar results hold for similar quantities with $u$ or their $x$-derivatives up to order $q-3$.
\end{lemma}


\begin{proof}
Lemma \ref{lem_reg} shows that 
\begin{equation}\label{dtu}
\|\partial_t u \|_{H^2} + \|\partial_{tt} u \|_{H^1}+ \|\partial_{ttt} u \| + \|\partial_t w \|_{H^2} + \|\partial_{tt} w \|_{H^1}+ \|\partial_{ttt} w \|  \le C,\quad \|\partial_t w\| + \|\partial_{tt} w \| \le  C\varepsilon,
\end{equation}
if the initial data is consistent up to order 3.

Taylor expansion of the first equation in \eqref{sch_exact} gives
\begin{equation}\begin{split}
& w(t_n)+c_i\Delta t\partial_t w(t_n) + O(\Delta t^2\partial_{tt}w) \\
= & w(t_n) -\Delta t \sum_{j=1}^{i-1}\tilde{a}_{i j} ((1-b^2)\partial_x u(t_n) - b\partial_x w(t_n)) + O(\Delta t^2 \partial_t\partial_x u)  + O(\Delta t^2 \partial_t\partial_x w) \\
& - \frac{\Delta t}{\varepsilon}\sum_{j=1}^i a_{i j}w(t_n)  + O\left(\frac{\Delta t^2}{\varepsilon} \partial_t w\right) - E_W^{(i)}.
\end{split}\end{equation}
Since $(u,w)$ satisfies \eqref{eq1}, we see that the $O(\Delta t)$ terms are canceled due to $c_i = \sum_j a_{ij} = \sum_j \tilde{a}_{ij}$. Therefore we get $E_W^{(i)}=O(\Delta t^2)$ in view of \eqref{dtu}.

Taylor expansion of the second equation in \eqref{sch_exact} gives
\begin{equation}\begin{split}
& w(t_n) + \Delta t \partial_t w(t_n) + \frac{1}{2}\Delta t^2 \partial_{tt} w(t_n) + O(\Delta t^3\partial_{ttt}w) \\
= & w(t_n) -\Delta t \sum_{j=1}^{s}\tilde{b}_{j} \Big((1-b^2)(\partial_x u(t_n)+c_j\Delta t \partial_t\partial_x u(t_n)) - b(\partial_x w(t_n)+c_j\Delta t\partial_t\partial_x w(t_n))\Big) \\ & + O(\Delta t^3 \partial_{tt}\partial_x u)  + O(\Delta t^3 \partial_{tt}\partial_x w) \\
& - \frac{\Delta t}{\varepsilon}\sum_{j=1}^s b_j(w(t_n)+c_j\Delta t\partial_t w(t_n))  + O\left(\frac{\Delta t^3}{\varepsilon} \partial_{tt} w\right) - E_W^{n+1}.
\end{split}\end{equation}
Similar as before, we may use the assumed order conditions to show that all the $O(\Delta t)$ and $O(\Delta t^2)$ terms are canceled. Therefore we get $E_W^{n+1}=O(\Delta t^3)$ in view of \eqref{dtu}.

The results for $u$ and their $x$-derivatives can be obtained similarly.

\end{proof}

\subsection{Energy estimates for the error}\label{sec_second2}

We first state a technical lemma concerning the matrix $(I+\mu A)^{-1}$.

\begin{lemma}\label{lem_mu}
Let $\mu>0$, possibly large or small. Then componentwise we have
\begin{equation}\label{lem_mu_1}
(I+\mu A)^{-1} = O(1),
\end{equation}
\begin{equation}\label{lem_mu_2}
A(I+\mu A)^{-1} = O\left(\frac{1}{1+\mu}\right).
\end{equation}
Furthermore, if {\bf (A)} is assumed, then
\begin{equation}\label{lem_mu_3}
 \big(A (I+\mu A)^{-1}\big)_{s,i} =  O\left(\frac{1}{(1+\mu)^2}\right),\quad i=1,\dots,s-1.
\end{equation}
\end{lemma}

\begin{proof}

Denote
\begin{equation}
D = I + \mu \text{diag}\{0,a_{22},\dots,a_{ss}\},\quad L = I+\mu A - D,
\end{equation}
as the diagonal and lower triangular parts of $I+\mu A$. Then 
\begin{equation}\label{muA0}\begin{split}
(I+\mu A)^{-1} = & (D+L)^{-1} = (I+D^{-1}L )^{-1}D^{-1} 
=  \Big(I + \sum_{k=1}^{s-1} (-1)^k (D^{-1}L )^k\Big)D^{-1}.
\end{split}\end{equation}
Notice that
\begin{equation}
D^{-1}L = \begin{pmatrix}
0 & 0 & 0 & \cdots & 0 \\
\frac{\mu}{1+\mu a_{22}}a_{21} & 0 & 0 & \cdots & 0 \\
\frac{\mu}{1+\mu a_{33}}a_{31} & \frac{\mu}{1+\mu a_{33}}a_{32} & 0 & \cdots & 0 \\
\cdots & \cdots & \cdots & \cdots & \cdots \\
\frac{\mu}{1+\mu a_{ss}}a_{s1} & \frac{\mu}{1+\mu a_{ss}}a_{s2} & \frac{\mu}{1+\mu a_{ss}}a_{s3} & \cdots & 0 \\
\end{pmatrix} =O\left(\frac{\mu}{1+\mu}\right) \le O(1),
\end{equation}
and $D^{-1}=O(1)$. Therefore \eqref{lem_mu_1} follows. 

Then we write
\begin{equation}\label{muA}\begin{split}
\mu A(I+\mu A)^{-1} = & (I+\mu A - I)(I+\mu A)^{-1} = I - (I+\mu A)^{-1} \\
= & I - D^{-1} - \sum_{k=1}^{s-1} (-1)^k (D^{-1}L )^kD^{-1}.
\end{split}\end{equation}
The previous analysis of $D^{-1} L$ shows that the last summation is $O(\frac{\mu}{1+\mu})$. We also have 
\begin{equation}
I - D^{-1} = \text{diag}\Big\{0,1-\frac{1}{1+\mu a_{22}},\dots,1-\frac{1}{1+\mu a_{ss}}\Big\} = O\left(\frac{\mu}{1+\mu}\right).
\end{equation}
Therefore \eqref{lem_mu_2} follows. 

Finally we prove \eqref{lem_mu_3}. From \eqref{muA}, it is clear that \eqref{lem_mu_3} holds for $i=2,\dots,s-1$ because these off-diagonal terms only comes from the summation $- \sum_{k=1}^{s-1} (-1)^k (D^{-1}L )^kD^{-1}
$, and the last multiplication by $D^{-1}$ gives an extra factor $\frac{1}{1+\mu a_{ii}}$. To see that the same is true for $i=1$, we notice that $A(I+\mu A)^{-1}\vec{v}=(I+\mu A)^{-1}A\vec{v}=0$ where $\vec{v}$ is in condition {\bf (A)}. It is clear that $v_1\ne 0$ since $A(2:s,2:s)$ is invertible. Then, by {\bf (A)}, we have
\begin{equation}
\sum_{i=1}^{s-1}\big(A (I+\mu A)^{-1})\big)_{s,i} v_i = 0,
\end{equation}
which implies \eqref{lem_mu_3} for $i=1$ from \eqref{lem_mu_3} for other $i$.
\end{proof}

Denote
\begin{equation}
\mu=\frac{\Delta t}{\varepsilon}.
\end{equation}
We use the vector notation 
\begin{equation}\label{vecuw}
\vec{w} = (w(t_n+c_1\Delta t),\dots,w(t_n+c_s\Delta t))^\top,\quad \vec{E}_W = (E_{W}^{(1)},\dots, E_{W}^{(s)})^\top.
\end{equation}
Then, noticing that $c_1=0$, we may write the first equation of \eqref{sch_exact} (together with its $u$-counterpart) as
\begin{equation}\label{sch_vec_exact}\begin{split}
& \vec{u} = u(t_n)\vec{e} -\Delta t \tilde{A} (b\partial_x \vec{u} + \partial_x \vec{w}) - \vec{E}_U,  \\
& \vec{w} = w(t_n)\vec{e} -\Delta t \tilde{A} ((1-b^2)\partial_x \vec{u} - b\partial_x \vec{w})  - \mu A \vec{w} - \vec{E}_W.\\
\end{split}\end{equation}

Denote
\begin{equation}
\vec{W}_e = \vec{W} - \vec{w},
\end{equation}
as the vector of numerical error in the $n$-th time step (with $\vec{W}$ given in \eqref{vecUW} and $\vec{w}$ given in \eqref{vecuw}, and similarly for $u$). We subtract \eqref{sch_vec_exact} with \eqref{sch_vec} and get
\begin{equation}\label{sch_vec_error}\begin{split}
& \vec{U}_e = U^{(1)}_e\vec{e} -\Delta t \tilde{A} (b\partial_x \vec{U}_e + \partial_x \vec{W}_e) + \vec{E}_U,  \\
& \vec{W}_e = W^{(1)}_e\vec{e} -\Delta t \tilde{A} ((1-b^2)\partial_x \vec{U}_e - b\partial_x \vec{W}_e)  - \mu A \vec{W}_e + \vec{E}_W. \\
\end{split}\end{equation}

Now we absorb the error vectors by introducing
\begin{equation}
\vec{U}_{e*} :=  \vec{U}_e - \vec{E}_U,\quad (I+\mu  A)\vec{W}_{e*} := (I+\mu  A)\vec{W}_e - \vec{E}_W,
\end{equation}
i.e.,
\begin{equation}
\vec{U}_{e} =  \vec{U}_{e*} + \vec{E}_U,\quad \vec{W}_{e} = \vec{W}_{e*} + (I+\mu  A)^{-1} \vec{E}_W.
\end{equation}
\eqref{sch_vec_error} can be written as
\begin{equation}\label{sch_vec_st}\begin{split}
& \vec{U}_{e*} = U^{(1)}_e\vec{e} -\Delta t \tilde{A} (b\partial_x \vec{U}_{e*} + \partial_x \vec{W}_{e*}) -\Delta t \tilde{A}\vec{F}_U,   \\
& \vec{W}_{e*} = W^{(1)}_e\vec{e} -\Delta t \tilde{A} ((1-b^2)\partial_x \vec{U}_{e*} - b\partial_x \vec{W}_{e*})  - \mu  A \vec{W}_{e*} -\Delta t \tilde{A}\vec{F}_W, \\
\end{split}\end{equation}
where we denote
\begin{equation}\label{FUW}\begin{split}
& \vec{F}_U =  b\partial_x \vec{E}_U + (I+\mu  A)^{-1}\partial_x \vec{E}_W,  \\
& \vec{F}_W = (1-b^2)\partial_x \vec{E}_U - b(I+\mu  A)^{-1}\partial_x \vec{E}_W.  \\
\end{split}\end{equation}

Notice that the first component of $\vec{U}_{e*}$ is exactly $U^{(1)}_e=U_e^n$. Multiplying the $\vec{U}_{e*},\vec{W}_{e*}$ equations by $\vec{U}_{e*}^\top M,\vec{W}_{e*}^\top M$  respectively and integrating in $x$, using Lemma~\ref{lem_v1} we get
\begin{subequations}\label{sch_vec_st1}\begin{align}
 \nonumber\textbf{MAIN} & \textbf{ ENERGY ESTIMATE 1: from $U^n_e$ to $U_{e*}^{(s)}$} \\
 \label{sch_vec_st1a} \|U_{e*}^{(s)}\|^2 \le & \|U^{n}_e\|^2 -  c \|\delta \vec{U}_{e*}\|^2 -\Delta t \int \vec{U}_{e*}^\top M \tilde{A} (b\partial_x \vec{U}_{e*} + \partial_x \vec{W}_{e*})\rd{x}  -\Delta t \int \vec{U}_{e*}^\top M \tilde{A}\vec{F}_U\rd{x},  \\
 \nonumber\|W_{e*}^{(s)}\|^2 \le & \|W^{n}_e\|^2 -  c \|\delta \vec{W}_{e*}\|^2 -\Delta t \int \vec{W}_{e*}^\top M\tilde{A} ((1-b^2)\partial_x \vec{U}_{e*} - b\partial_x \vec{W}_{e*}) \rd{x}  \\ 
&\label{sch_vec_st1b} -\Delta t \int \vec{W}_{e*}^\top M \tilde{A}\vec{F}_W\rd{x} - \mu  \int \vec{W}_{e*}^\top M A \vec{W}_{e*}\rd{x}.
\end{align}\end{subequations}

The error of $U^{n+1},W^{n+1}$ satisfies
\begin{equation}\begin{split}
& U^{n+1}_e = U^n_e -\Delta t \sum_{j=1}^{s} \tilde{b}_j (b\partial_x U^{(j)}_e + \partial_x W^{(j)}_e) + E^{n+1}_U,\\
& W^{n+1}_e = W^n_e -\Delta t \sum_{j=1}^{s}\tilde{b}_j ((1-b^2)\partial_x U^{(j)}_e - b\partial_x W^{(j)}_e) - \mu \sum_{j=1}^s a_{s j} W^{(j)}_e  + E^{n+1}_W. \\
\end{split}\end{equation}
We use vector notation and rewrite it with $U_{e*},W_{e*}$ as
\begin{equation}\label{sch_n1_e}\begin{split}
U^{n+1}_e = & U^n_e -\Delta t \vec{\tilde{b}} (b\partial_x \vec{U}_{e*} + \partial_x \vec{W}_{e*}) -\Delta t \vec{\tilde{b}} \vec{F}_U + E^{n+1}_U,\\
W^{n+1}_e = & W^n_e -\Delta t \vec{\tilde{b}} ((1-b^2)\partial_x \vec{U}_{e*} - b\partial_x \vec{W}_{e*}) - \mu A_s \vec{W}_{e*}  \\
& -\Delta t \vec{\tilde{b}} \vec{F}_W  - \mu A_s (I+\mu  A)^{-1} \vec{E}_W + E^{n+1}_W, \\
\end{split}\end{equation}
where $A_s$ denotes the last row of the matrix $A$ (and similar notation is used for the last row of other matrices).

Subtracting with the last rows of the vector equations \eqref{sch_vec_st}, we get
\begin{equation}\label{energy_e_n0}\begin{split}
& U^{n+1}_e = U^{(s)}_{e*} -\Delta t (\vec{\tilde{b}}-\tilde{A}_s) (b\partial_x \vec{U}_{e*} + \partial_x \vec{W}_{e*}) -\Delta t (\vec{\tilde{b}}-\tilde{A}_s) \vec{F}_U + E^{n+1}_U,\\
& W^{n+1}_e = W^{(s)}_{e*} -\Delta t (\vec{\tilde{b}}-\tilde{A}_s) ((1-b^2)\partial_x \vec{U}_{e*} - b\partial_x \vec{W}_{e*})   -\Delta t(\vec{\tilde{b}}-\tilde{A}_s) \vec{F}_W  - \mu A_s (I+\mu  A)^{-1} \vec{E}_W + E^{n+1}_W. \\
\end{split}\end{equation}
Then we do energy estimate similarly to \eqref{energyn1}: multiply by $2U^{n+1}_e,2W^{n+1}_e$ respectively and integrate, and add with \eqref{sch_vec_st1}. This gives
\begin{subequations}\label{energy_e_n01}\begin{align}
\nonumber\textbf{MAIN} & \textbf{ ENERGY ESTIMATE 2: from  $U_{e*}^{(s)}$ to $U^{n+1}_e$} \\
 \nonumber\|U^{n+1}_e\|^2 = & \|U^{(s)}_{e*}\|^2 - \|U^{n+1}_e-U^{(s)}_{e*}\|^2 -2\Delta t \int U^{n+1}_e(\vec{\tilde{b}}-\tilde{A}_s) (b\partial_x \vec{U}_{e*} + \partial_x \vec{W}_{e*})\rd{x} \\
& \label{energy_e_n01a} -2\Delta t \int U^{n+1}_e(\vec{\tilde{b}}-\tilde{A}_s) \vec{F}_U\rd{x} + 2\int U^{n+1}_e E^{n+1}_U\rd{x},\\
 \nonumber \|W^{n+1}_e\|^2 = & \|W^{(s)}_{e*}\|^2 - \|W^{n+1}_e-W^{(s)}_{e*}\|^2 -2\Delta t \int W^{n+1}_e(\vec{\tilde{b}}-\tilde{A}_s) ((1-b^2)\partial_x \vec{U}_{e*} - b\partial_x \vec{W}_{e*}) \rd{x}  \\
&  \label{energy_e_n01b} -2\Delta t\int W^{n+1}_e(\vec{\tilde{b}}-\tilde{A}_s) \vec{F}_W\rd{x} + 2\int W^{n+1}_e E^{n+1}_W\rd{x}  - \underline{2\int W^{n+1}_e \mu A_s (I+\mu  A)^{-1} \vec{E}_W\rd{x}} . 
\end{align}\end{subequations}

\begin{remark}
To handle the low stage order of intermediate stages, the main technique we use here is the auxiliary error vector $\vec{W}_{e*}$. Comparing \eqref{sch_vec_st} with \eqref{sch_vec_error}, we see that the extra terms involving $\vec{E}_W$ in \eqref{sch_vec_st} has an extra $\Delta t$ factor in front, making the influence of these terms smaller. This enables us to conduct the energy estimate for $W_{e*}^{(s)}$ easily via \eqref{sch_vec_st1}. In fact, we will see in the next section (c.f., \eqref{sch_vec_stst}) that this trick can be applied more than once, if one aims to study the uniform accuracy of higher order. All the difficulty involving the lower stage order is finally unwrapped in \eqref{energy_e_n01}, in which we recover $W^{n+1}_e$ from $W_{e*}^{(s)}$. The underlined term therein is the most difficult term, and we will analyze it by using the delicate properties of the matrix $(I+\mu  A)^{-1}$ as stated in Lemma \ref{lem_mu}.
\end{remark}

\subsection{Combined energy estimate: prove first order uniform accuracy}\label{sec_second3}

We will first prove the first order uniform accuracy of the scheme in this subsection, and then improve it to second order uniform accuracy in the next subsection.

We take $(1-b^2)\eqref{sch_vec_st1a}+\eqref{sch_vec_st1b}$, and the same energy estimate as in \eqref{vec_energy3} gives
\begin{equation}\label{energy_e_st}\begin{split}
 (1-b^2)\|U_{e*}^{(s)}\|^2+\|W_{e*}^{(s)}\|^2 
\le & (1+C\Delta t)((1-b^2)\|U^{n}_e\|^2+\|W^{n}_e\|^2)  -  c \|\delta \vec{U}_{e*}\|^2 -  c \|\delta \vec{W}_{e*}\|^2  \\
& -\Delta t (1-b^2) \int \vec{U}_{e*}^\top M \tilde{A}\vec{F}_U \rd{x}  -\Delta t  \int \vec{W}_{e*}^\top M \tilde{A}\vec{F}_W \rd{x}, \\
\end{split}\end{equation}
where $\vec{F}_U$ and $\vec{F}_W$ consist of linear combinations of $\partial_x E_U^{(i)}$ and $\partial_x E_W^{(i)}$ with $O(1)$ coefficients (due to \eqref{lem_mu_1}). To treat the terms with $\vec{F}_U$ and $\vec{F}_W$, we have the estimate
\begin{equation}\label{cons4}
\Delta t\left| \int U_{e*}^{(i)} \partial_x E_{U}^{(j)}\rd{x} \right| \le \Delta t (\|U_{e*}^{(i)}\|^2 + \|\partial_x E_{U}^{(j)}\|^2) \le C\Delta t \|U_e^n\|^2 + C\Delta t \|U_{e*}^{(i)}-U_e^n\|^2 + C\Delta t^5,
\end{equation}
since $\partial_x E_{U}^{(j)}= O(\Delta t^2)$ for every $j$ by Lemma \ref{lem_LTE2} (where consistency of initial data up to order 4 is used), and similarly for terms with $W$. Therefore, absorbing $C\Delta t \|U_{e*}^{(i)}-U_e^n\|^2$ by the good term $ -  c \|\delta \vec{U}_{e*}\|^2$ for $\Delta t \le c$, we get
\begin{equation}\begin{split}
& (1-b^2)\|U_{e*}^{(s)}\|^2+\|W_{e*}^{(s)}\|^2 
\le  (1+C\Delta t)((1-b^2)\|U^{n}_e\|^2+\|W^{n}_e\|^2) -  c \|\delta \vec{U}_{e*}\|^2 -  c \|\delta \vec{W}_{e*}\|^2 + O(\Delta t^5).
\end{split}\end{equation}

We take $(1-b^2)\eqref{energy_e_n01a}+\eqref{energy_e_n01b}$, conducting a similar energy estimate as in \eqref{energyn1}, and adding with the above estimate gives
\begin{equation}\label{energy_e_n1}\begin{split}
& (1-b^2)\|U_{e}^{n+1}\|^2+\|W_{e}^{n+1}\|^2 \\
\le&  (1+C\Delta t)((1-b^2)\|U^{n}_e\|^2+\|W^{n}_e\|^2) -  c \|\delta \vec{U}_{e*}\|^2 -  c \|\delta \vec{W}_{e*}\|^2 - c\|U_{e}^{n+1}-U^{(s)}_{e*}\|^2 - c\|W_{e}^{n+1}-W^{(s)}_{e*}\|^2 \\
& -2\Delta t (1-b^2)\int  U_{e}^{n+1}(\vec{\tilde{b}}-\tilde{A}_s) \vec{F}_U\rd{x} + 2\int  (1-b^2)U_{e}^{n+1} E^{n+1}_U\rd{x} -2\Delta t \int  W_{e}^{n+1}(\vec{\tilde{b}}-\tilde{A}_s) \vec{F}_W\rd{x} \\
&+ 2\int  W_{e}^{n+1}E^{n+1}_W\rd{x}    \underline{ - 2 \int  W_{e}^{n+1}\mu A_s (I+\mu  A)^{-1} \vec{E}_W\rd{x}}+ O(\Delta t^5). \\
\end{split}\end{equation}
In the terms involving $\vec{F}_U$ and $\vec{F}_W$, one can estimate as
\begin{equation}
\Delta t \left|\int  U_{e}^{n+1} \partial_x E^{(j)}_U\rd{x}\right| \le c\Delta t \|U_{e}^{n+1}\|^2 + C\Delta t \|\partial_x E^{(j)}_U\|^2\le c\Delta t \|U_{e}^{n+1}\|^2 + O(\Delta t^5),
\end{equation}
and the term $c\Delta t \|U_{e}^{n+1}\|^2$ can be absorbed by LHS. 

The term $\int  U_{e}^{n+1} E^{n+1}_U\rd{x}$ (and the similar one with $W$) can be estimated similarly using $E^{n+1}_U=O(\Delta t^3)$. Therefore all these non-stiff terms give a contribution of at most $O(\Delta t^5)$.

The worst term is the stiff (underlined) term $\int  W_{e}^{n+1}\mu A_s (I+\mu A)^{-1} \vec{E}_W\rd{x}$. For the matrix $\mu A(I+\mu A)^{-1}$, the best one can say is that it has elements $\le C\frac{\mu}{1+\mu}$ from \eqref{lem_mu_2}. Therefore this term can be bounded by 
\begin{equation}
\left|\int  W_{e}^{n+1}\mu A_s (I+\mu A)^{-1} \vec{E}_W\rd{x}\right| \le c\Delta t \|W_{e}^{n+1}\|^2 + \frac{C}{\Delta t}\|\mu A_s (I+\mu A)^{-1} \vec{E}_W\|^2\le c\Delta t \|W_{e}^{n+1}\|^2 + C\Delta t^3\frac{\mu^2}{(1+\mu)^2},
\end{equation}
using $E^{(i)}_W = O(\Delta t^2)$ from Lemma \ref{lem_LTE2}, and $c\Delta t \|W_{e}^{n+1}\|^2$ can be absorbed by LHS. Therefore we finally get
\begin{equation}\begin{split}
& (1-b^2)\|U_{e}^{n+1}\|^2+\|W_{e}^{n+1}\|^2 
\le  (1+C\Delta t)((1-b^2)\|U^{n}_e\|^2+\|W^{n}_e\|^2) + O\Big(\Delta t^5+\Delta t^3\frac{\mu^2}{(1+\mu)^2}\Big). \\
\end{split}\end{equation}
Using Gronwall inequality, we get
\begin{equation}\label{final_est1}\begin{split}
& (1-b^2)\|U_{e}^n\|^2+\|W_{e}^n\|^2 \le C(T)\Big(\Delta t^4+\Delta t^2\frac{\mu^2}{(1+\mu)^2}\Big),\quad \forall n\Delta t \le T.
\end{split}\end{equation}

Notice that the same error estimate works for any intermediate stages $U^{n,(j)}_{e*},W^{n,(j)}_{e*}$. Also, the above estimate only utilizes the consistency of initial data up to order 4 (c.f. \eqref{cons4}). Since we assumed the consistency of initial data up to order 6, the same error estimate works for the $x$-derivatives of these quantities up to order 2. This will be used in the next subsection.

\begin{remark}
Notice that \eqref{final_est1} implies first order uniform accuracy. To be more precise, if $\varepsilon=O(1)$, i.e., $\mu=O(\Delta t)$, then it gives second order accuracy, but it degenerates to first order accuracy for large $\mu$, i.e., small $\varepsilon$. This motivates us to utilize the last term in \eqref{sch_vec_st1b}, a coercive term proportional to $\mu$, to study the second order uniform accuracy in the next subsection.
\end{remark}

Since  \eqref{final_est1} already implies Theorem \ref{thm_2nd} in the case of $\varepsilon=O(1)$, we may assume $\varepsilon\le 1$ in the rest of this proof. Thus $\Delta t^2\frac{\mu^2}{(1+\mu)^2}$ is always the worst term above since $\Delta t\le c$ is assumed.

\subsection{Improve to second order}\label{sec_second4}

We then improve the error estimate \eqref{final_est1} to uniform second order for $W$. We start by revisiting \eqref{sch_vec_st1b}. Since {\bf (M1)} and {\bf (A)} are assumed, Lemma \ref{lem_v} gives the coercive estimate 
\begin{equation}
\mu\int \vec{W}_{e*}^\top M A \vec{W}_{e*}\rd{x} \ge c\mu \|W_{e*}^{(s)}\|^2,
\end{equation}
which provides a good term $-c\mu \|W_{e*}^{(s)}\|^2$ in \eqref{sch_vec_st1b}. We estimate the integral in the first line of \eqref{sch_vec_st1b} as
\begin{equation}\label{We_energy_0}\begin{split}
\Delta t\left| \int W^{(i)}_{e*}\partial_x U^{(j)}_{e*}\rd{x}\right| \le & c\frac{\mu}{1+\mu}\|W^{(i)}_{e*}\|^2 + C\Delta t^2\frac{1+\mu}{\mu}\|\partial_x U^{(j)}_{e*}\|^2
\le  c\frac{\mu}{1+\mu} \|W^{(i)}_{e*}\|^2 + C\Delta t^4\frac{\mu}{1+\mu}, \\
\end{split}\end{equation}
by using \eqref{final_est1} for $\|\partial_x U^{(j)}_{e*}\|^2$, and similarly for the term involving $\partial_x W^{(j)}_{e*}$. Here $c\frac{\mu}{1+\mu} \|W^{(i)}_{e*}\|^2 $ can be absorbed by $-c\mu \|W_{e*}^{(s)}\|^2$ together with good term $-c\|\delta \vec{W}_{e*}\|^2$ in \eqref{sch_vec_st1b} since
\begin{equation}
\|\delta \vec{W}_{e*}\|^2 + \mu \|W_{e*}^{(s)}\|^2 \ge \frac{\mu}{1+\mu}(\|W_{e*}^{(i)}-W_{e*}^{(s)}\|^2 +  \|W_{e*}^{(s)}\|^2) \ge \frac{\mu}{2(1+\mu)}\|W^{(i)}_{e*}\|^2.
\end{equation}

The first integral in the second line of \eqref{sch_vec_st1b} can be easily controlled by $C\Delta t \|W^n_e\|^2 + C\Delta t \|W^{(i)}_{e*}-W^n_e\|^2 + O(\Delta t^5)$ as we did for \eqref{energy_e_st}. Therefore we get
\begin{equation}\label{We_energy}\begin{split}
&\|W_{e*}^{(s)}\|^2 \le (1+C\Delta t)\|W^{n}_e\|^2 + C\Delta t^4\frac{\mu}{1+\mu} - c  \|\delta \vec{W}_{e*}\|^2 - c\mu \|W^{(s)}_{e*}\|^2,
\end{split}\end{equation}
since we assumed $\Delta t\le c,\,\varepsilon\le 1$ and thus $O(\Delta t^5)\le O(\Delta t^4\frac{\mu}{1+\mu})$.

Adding with \eqref{energy_e_n01b}, we get
\begin{equation}\label{We_energy1}\begin{split}
 \|W_{e}^{n+1}\|^2 =& (1+C\Delta t)\|W^{n}_{e}\|^2 - c\|W_{e}^{n+1}-W^{(s)}_{e*}\|^2 - c  \|\delta \vec{W}_{e*}\|^2 - c\mu \|W^{(s)}_{e*}\|^2\\
& -2\Delta t \int W_{e}^{n+1}(\vec{\tilde{b}}-\tilde{A}_s) ((1-b^2)\partial_x \vec{U}_{e*} - b\partial_x \vec{W}_{e*}) \rd{x}   -2\Delta t \int W_{e}^{n+1}(\vec{\tilde{b}}-\tilde{A}_s) \vec{F}_W\rd{x} \\
 & + 2\int W_{e}^{n+1} E^{n+1}_W\rd{x} - \underline{ 2\int W_{e}^{n+1}\mu A_s (I+\mu A)^{-1} \vec{E}_W\rd{x}}  + C\Delta t^4\frac{\mu}{1+\mu},\\
\end{split}\end{equation}
which is analogous to \eqref{energy_e_n1}. Notice that 
\begin{equation}
\|W_{e}^{n+1}-W^{(s)}_{e*}\|^2 + \mu \|W^{(s)}_{e*}\|^2 \ge \frac{\mu}{1+\mu}(\|W_{e}^{n+1}-W^{(s)}_{e*}\|^2 + \|W^{(s)}_{e*}\|^2) \ge \frac{\mu}{2(1+\mu)}\|W^{n+1}_{e}\|^2,
\end{equation}
i.e., we gain a good term $-c\frac{\mu}{1+\mu} \|W^{n+1}_{e}\|^2$ out of the terms $- c\|W_{e}^{n+1}-W^{(s)}_{e*}\|^2 - c\mu \|W^{(s)}_{e*}\|^2$. This helps us improve the worst (underlined) term estimate as
\begin{equation}
\left|\int  W_{e}^{n+1} \mu A_s (I+\mu A)^{-1} \vec{E}_W\rd{x}\right| \le c\frac{\mu}{1+\mu} \|W^{n+1}_{e}\|^2 + C\frac{1+\mu}{\mu}\Delta t^4\cdot \frac{\mu^2}{(1+\mu)^2} = c\frac{\mu}{1+\mu} \|W^{n+1}_{e}\|^2 + C\Delta t^4 \frac{\mu}{1+\mu}.
\end{equation}
By estimating other terms as what we did for \eqref{energy_e_n1} with the term involving $\vec{U}_{e*}$ treated as in \eqref{We_energy_0} (which gives $O(\Delta t^5)$ in total), we obtain
\begin{equation}\begin{split}
& \|W_{e}^{n+1}\|^2 \le (1+C\Delta t)\|W^{n}_{e}\|^2 + \frac{\mu}{1+\mu}(C\Delta t^4 - c \|W^{n+1}_{e}\|^2).
\end{split}\end{equation}
Then a bootstrap argument gives
\begin{equation}\label{final_est2}\begin{split}
\|W_{e}^{n}\|^2 \le C(T)\Delta t^4,\quad \forall n\Delta t \le T,
\end{split}\end{equation}
i.e., uniform second order accuracy of $W$. 


Finally we improve the error estimate \eqref{final_est1} to uniform second order for $U$. In \eqref{sch_vec_st1a}, now we may estimate as
\begin{equation}
\left|\Delta t\int U^{(i)}_{e*}\partial_x W^{(j)}_{e*}\right| \le C\Delta t \|U^{(i)}_{e*}\|^2 + C\Delta t \|\partial_x W^{(j)}_{e*}\|^2\le C\Delta t \|U^n_e\|^2 +C\Delta t \|U^{(i)}_{e*}-U^n_e\|^2 + C\Delta t^5,
\end{equation}
using \eqref{final_est2} for $\partial_x W^{(j)}_{e*}$. By estimating other terms in the same way as before, we obtain
\begin{equation}\begin{split}
& \|U_{e*}^{(s)}\|^2 \le (1+C\Delta t)\|U^{n}_e\|^2 -  c \|\delta \vec{U}_{e*}\|^2 + C\Delta t^5.  \\
\end{split}\end{equation}
The same treatment can be applied to the term $\Delta t\int U^{n+1}_{e}\partial_x W^{(j)}_{e*}\rd{x}$ in  \eqref{energy_e_n01a}. By estimating other terms in the same way as before and adding it with the previous estimate, 
we obtain
\begin{equation}\begin{split}
& \|U_{e}^{n+1}\|^2 \le (1+C\Delta t)\|U^{n}_e\|^2  + C\Delta t^5.  \\
\end{split}\end{equation}
This gives the uniform second order accuracy of $U$ by the Gronwall inequality, and finishes the proof of Theorem \ref{thm_2nd}.


\section{Third order uniform accuracy}
\label{sec_third}

Our main result in this section is stated as follows.
\begin{theorem}[Third order uniform accuracy of IMEX-RK schemes]\label{thm_3rd}
Under the same assumptions as in Theorem \ref{thm_2nd}, further assume 
\begin{itemize}
\item The IMEX-RK scheme satisfies the standard third order conditions (eqn (9) in \cite{PR05}).
\item The stage order conditions
\begin{equation}\label{LTE_3rd_0}
\frac{1}{2}c_i^2 = \sum_{j=1}^{i-1}\tilde{a}_{i j}c_j = \sum_{j=1}^i a_{i j}c_j,\quad i = 3,\dots,s.
\end{equation}
\item The `vanishing coefficient condition'
\begin{equation}\label{cond_A0}
\text{$\tilde{b}_2=0$ and $a_{i,2}=0,\,i=3,\dots,s$.}
\end{equation}
\item The initial data is consistent up to order $8$.
\end{itemize}
Then for any $T>0$ and $n\in\mathbb{Z}_{\ge 0}$ with $n\Delta t\le T$, we have
\begin{equation}
(1-b^2)\|U^n_e\|^2+\|W^n_e\|^2 \le C\Big(\Delta t^6 + \frac{1}{N^{12}}\Big),
\end{equation}
with $C$ independent of $\varepsilon$, $N$ and $\Delta t$.
\end{theorem}

Among the IMEX-RK schemes considered in this paper, only the IMEX-RK scheme BHR(5,5,3)* given in the Appendix satisfies the assumptions in Theorem~\ref{thm_3rd}, hence it will exhibit at least third order uniform accuracy in time.  

We will prove this theorem in the rest of this section. Similar to Theorem \ref{thm_2nd}, we may handle the initial projection error by Lemma \ref{Fourier1} and ignore the $N$-dependence in the rest of the proof. The proof follows the same structure as that of Theorem \ref{thm_2nd} but more technical. 

\subsection{Local truncation error}

\begin{lemma}\label{lem_LTE3}
For a third order IMEX-RK scheme of type CK with $c_i=\tilde{c}_i,\,i=1,\dots,s$, assume it further satisfies condition (\ref{LTE_3rd_0}) and the initial data is consistent up to order $q\ge 4$. Then
\begin{equation}
E_W^{(2)} = O(\Delta t^2),\quad E_W^{(i)} = O(\Delta t^3),\,i=3,\dots,s,\quad E_W^{n+1} = O(\Delta t^4),
\end{equation}
and similar results hold for similar quantities with $u$ or their $x$-derivatives up to order $q-4$.
\end{lemma}
The proof of this lemma is similar to Lemma \ref{lem_LTE2} and thus omitted.



\subsection{Energy estimates for the error}

Starting from \eqref{sch_vec_st}, we do a further change of variable in order to absorb the last term involving $\vec{E}_U,\vec{E}_W$ in both equations:
\begin{equation}\begin{split}
& \vec{U}_{e**} =  \vec{U}_{e*} +\Delta t \tilde{A} \vec{F}_U,\\
& (I+\mu  A)\vec{W}_{e**} = (I+\mu  A)\vec{W}_{e*}  +\Delta t \tilde{A} \vec{F}_W, 
\end{split}\end{equation}
i.e.,
\begin{equation}\begin{split}
& \vec{U}_{e*} =  \vec{U}_{e**} -\Delta t \tilde{A} \vec{F}_U,\\
& \vec{W}_{e*} = \vec{W}_{e**}  -\Delta t(I+\mu  A)^{-1} \tilde{A} \vec{F}_W.
\end{split}\end{equation}
Substituting into \eqref{sch_vec_st}, we get
\begin{equation}\label{sch_vec_stst}\begin{split}
& \vec{U}_{e**} = U^{(1)}_e\vec{e} -\Delta t \tilde{A} \Big(b\partial_x \vec{U}_{e**} + \partial_x \vec{W}_{e**}\Big)  + \Delta t^2 \tilde{A} \vec{G}_U, \\
& \vec{W}_{e**} = W^{(1)}_e\vec{e} -\Delta t \tilde{A} \Big((1-b^2)\partial_x \vec{U}_{e**} - b\partial_x \vec{W}_{e**}\Big)  - \mu  A \vec{W}_{e**}  + \Delta t^2 \tilde{A} \vec{G}_W,
\end{split}\end{equation}
where we denote
\begin{equation}\begin{split}
\vec{G}_U = & b \tilde{A}\partial_x\vec{F}_U +  (I+\mu  A)^{-1} \tilde{A} \partial_x\vec{F}_W, \\
\vec{G}_W = & (1-b^2) \tilde{A}\partial_x \vec{F}_U  - b(I+\mu  A)^{-1} \tilde{A} \partial_x \vec{F}_W.
\end{split}\end{equation}
Recall the definition of $\vec{F}_U$ and $\vec{F}_W$ in \eqref{FUW}, we see that $\vec{G}_U$ and $\vec{G}_W$ consist of linear combinations of $\partial_{xx} E_U^{(i)}$ and $\partial_{xx} E_W^{(i)}$ with $O(1)$ coefficients (due to \eqref{lem_mu_1}).

Multiplying the $\vec{U}_{e**},\vec{W}_{e**}$ equations by $\vec{U}_{e**}^\top M,\vec{W}_{e**}^\top M$  respectively and integrate, we get
\begin{subequations}\label{sch_vec_stst1}\begin{align}
\nonumber \textbf{MAIN} & \textbf{ ENERGY ESTIMATE 1: from $U^n_e$ to $U_{e**}^{(s)}$} \\
\label{sch_vec_stst1a} \|U_{e**}^{(s)}\|^2 \le & \|U^{n}_e\|^2 -  c \|\delta \vec{U}_{e**}\|^2  -\Delta t \int \vec{U}_{e**}^\top M \tilde{A} (b\partial_x \vec{U}_{e**} + \partial_x \vec{W}_{e**})\rd{x}  +\Delta t^2 \int \vec{U}_{e**}^\top M\tilde{A} \vec{G}_U \rd{x},   \\
\nonumber \|W_{e**}^{(s)}\|^2 \le & \|W^{n}_e\|^2 -  c \|\delta \vec{W}_{e**}\|^2 -\Delta t \int \vec{W}_{e**}^\top M\tilde{A} ((1-b^2)\partial_x \vec{U}_{e**} - b\partial_x \vec{W}_{e**}) \rd{x} \\
\label{sch_vec_stst1b} & - \mu  \int \vec{W}_{e**}^\top M A \vec{W}_{e**}\rd{x}  +\Delta t^2 \int \vec{W}_{e**}^\top M\tilde{A} \vec{G}_W\rd{x}.
\end{align}\end{subequations}

Then we rewrite \eqref{sch_n1_e} with $U_{e**},W_{e**}$ as
\begin{equation}\begin{split}
U^{n+1}_e = & U^n_e -\Delta t \vec{\tilde{b}} \Big(b\partial_x \vec{U}_{e**} + \partial_x \vec{W}_{e**}\Big) -\Delta t \vec{\tilde{b}} \vec{F}_U + E^{n+1}_U  + \Delta t^2 \vec{\tilde{b}} \vec{G}_U, \\
W^{n+1}_e = & W^n_e -\Delta t \vec{\tilde{b}} \Big((1-b^2)\partial_x \vec{U}_{e**} - b\partial_x \vec{W}_{e**}\Big) - \mu A_s \vec{W}_{e**}   -\Delta t \vec{\tilde{b}} \vec{F}_W  - \mu A_s (I+\mu  A)^{-1} \vec{E}_W + E^{n+1}_W \\
& + \Delta t^2 \vec{\tilde{b}} \vec{G}_W  + \Delta t \mu A_s(I+\mu  A)^{-1} \tilde{A} \vec{F}_W.
\end{split}\end{equation}

Subtracting with the last row of the vector equation \eqref{sch_vec_stst}, we get
\begin{equation}\label{energy_e_n0stst}\begin{split}
U^{n+1}_e = & U^{(s)}_{e**} -\Delta t (\vec{\tilde{b}}-\tilde{A}_s) (b\partial_x \vec{U}_{e**} + \partial_x \vec{W}_{e**}) -\Delta t \vec{\tilde{b}} \vec{F}_U + E^{n+1}_U  + \Delta t^2 (\vec{\tilde{b}}-\tilde{A}_s) \vec{G}_U, \\
W^{n+1}_e = & W^{(s)}_{e**} -\Delta t (\vec{\tilde{b}}-\tilde{A}_s) ((1-b^2)\partial_x \vec{U}_{e**} - b\partial_x \vec{W}_{e**})    -\Delta t \vec{\tilde{b}} \vec{F}_W  - \mu A_s (I+\mu  A)^{-1} \vec{E}_W + E^{n+1}_W \\
& + \Delta t^2 (\vec{\tilde{b}}-\tilde{A}_s) \vec{G}_W  + \Delta t \mu A_s(I+\mu  A)^{-1} \tilde{A} \vec{F}_W.
\end{split}\end{equation}
Then we do energy estimate similar to \eqref{energyn1}: multiplying by $2U^{n+1}_e,2W^{n+1}_e$ respectively and integrate, and adding with \eqref{sch_vec_st1}, we get
\begin{subequations}\label{energy_e_n01stst}\begin{align}
\nonumber \textbf{MAIN} & \textbf{ ENERGY ESTIMATE 2: from  $U_{e**}^{(s)}$ to $U^{n+1}_e$} \\
\nonumber \|U^{n+1}_e\|^2 = & \|U^{(s)}_{e**}\|^2 -  \|U^{n+1}_e-U^{(s)}_{e**}\|^2 - 2\Delta t \int U^{n+1}_e (\vec{\tilde{b}}-\tilde{A}_s) (b\partial_x \vec{U}_{e**} + \partial_x \vec{W}_{e**})\rd{x} \\
\label{energy_e_n01ststa} & -2\Delta t \int U^{n+1}_e\vec{\tilde{b}} \vec{F}_U\rd{x} + 2\int U^{n+1}_e E^{n+1}_U\rd{x} + 2\Delta t^2 \int U^{n+1}_e(\vec{\tilde{b}}-\tilde{A}_s) \vec{G}_U\rd{x}, \\
\nonumber \|W^{n+1}_e\|^2 = & \|W^{(s)}_{e**}\|^2 - \|W^{n+1}_e-W^{(s)}_{e**}\|^2  -2\Delta t \int W^{n+1}_e (\vec{\tilde{b}}-\tilde{A}_s) ((1-b^2)\partial_x \vec{U}_{e**} - b\partial_x \vec{W}_{e**})\rd{x}   \\
\nonumber & -2\Delta t  \int W^{n+1}_e \vec{\tilde{b}} \vec{F}_W\rd{x}  -  2\int W^{n+1}_e \mu A_s (I+\mu  A)^{-1} \vec{E}_W\rd{x} +  2\int W^{n+1}_e E^{n+1}_W\rd{x} \\
\label{energy_e_n01ststb} & + 2\Delta t^2  \int W^{n+1}_e (\vec{\tilde{b}}-\tilde{A}_s) \vec{G}_W\rd{x} + 2\Delta t  \int W^{n+1}_e \mu A_s(I+\mu  A)^{-1} \tilde{A} \vec{F}_W\rd{x}.
\end{align}\end{subequations}

\subsection{Combined energy estimate: prove second order uniform accuracy}

We will first prove the second order uniform accuracy of the scheme in this subsection, and then improve it to third order uniform accuracy in the next subsection.

We take $(1-b^2)\eqref{sch_vec_stst1a} + \eqref{sch_vec_stst1b}$. The same energy estimate as in \eqref{vec_energy3} gives
\begin{equation}\label{energy_e_stst}\begin{split}
& (1-b^2)\|U_{e**}^{(s)}\|^2 + \|W_{e**}^{(s)}\|^2 \le (1-b^2)\|U^{n}_e\|^2 + \|W^{n}_e\|^2 -  c \|\delta \vec{U}_{e**}\|^2 -  c \|\delta \vec{W}_{e**}\|^2  \\
& +\Delta t^2 (1-b^2)\int \vec{U}_{e**}^\top M\tilde{A} \vec{G}_U \rd{x}   +\Delta t^2 \int \vec{W}_{e**}^\top M\tilde{A} \vec{G}_W\rd{x}.
\end{split}\end{equation}
To treat the terms with $\vec{G}_U$ and $\vec{G}_W$, we have the estimate
\begin{equation}
\Delta t^2\left| \int U_{e**}^{(i)} \partial_{xx} E_{U}^{(j)} \right| \le \Delta t^2 \Big(\frac{1}{\Delta t}\|U_{e**}^{(i)}\|^2 +\Delta t \|\partial_{xx} E_{U}^{(j)}\|^2\Big) \le C\Delta t \|U_e^n\|^2 + C\Delta t \|U_{e**}^{(i)}-U_e^n\|^2 + C\Delta t^7,
\end{equation}
since $\partial_{xx} E_{U}^{(j)}= O(\Delta t^2)$ for every $j$ by Lemma \ref{lem_LTE3}, and similar for terms with $W$. Therefore we get
\begin{equation}\begin{split}
& (1-b^2)\|U_{e**}^{(s)}\|^2+\|W_{e**}^{(s)}\|^2 
\le  (1+C\Delta t)((1-b^2)\|U^{n}_e\|^2+\|W^{n}_e\|^2) -  c \|\delta \vec{U}_{e**}\|^2 -  c \|\delta \vec{W}_{e**}\|^2 + O(\Delta t^7).
\end{split}\end{equation}

Taking $(1-b^2)\eqref{energy_e_n01ststa}+\eqref{energy_e_n01ststb}$ and adding with the above estimate gives
\begin{equation}\label{energy_e_n1stst}\begin{split}
& (1-b^2)\|U^{n+1}_e\|^2+\|W^{n+1}_e\|^2 \\
= & (1+C\Delta t)((1-b^2)\|U^{n}_e\|^2+\|W^{n}_e\|^2) - c \|U^{n+1}_e-U^{(s)}_{e**}\|^2 - c\|W^{n+1}_e-W^{(s)}_{e**}\|^2 -  c \|\delta \vec{U}_{e**}\|^2 -  c \|\delta \vec{W}_{e**}\|^2\\
& - \Delta t (1-b^2)\int U^{n+1}_e (\vec{\tilde{b}}-\tilde{A}_s) (b\partial_x \vec{U}_{e**} + \partial_x \vec{W}_{e**})\rd{x}   + (1-b^2)\int U^{n+1}_e E^{n+1}_U\rd{x}\\
&  + \Delta t^2 (1-b^2)\int U^{n+1}_e(\vec{\tilde{b}}-\tilde{A}_s) \vec{G}_U\rd{x} -\underline{\Big[(1-b^2)\Delta t \int U^{n+1}_e\vec{\tilde{b}} \vec{F}_U\rd{x}+\Delta t  \int W^{n+1}_e \vec{\tilde{b}} \vec{F}_W\rd{x}\Big]} \\
& -\Delta t \int W^{n+1}_e (\vec{\tilde{b}}-\tilde{A}_s) ((1-b^2)\partial_x \vec{U}_{e**} - b\partial_x \vec{W}_{e**})\rd{x}     -  \underline{\int W^{n+1}_e \mu A_s (I+\mu  A)^{-1} \vec{E}_W\rd{x}} \\
& +  \int W^{n+1}_e E^{n+1}_W\rd{x}  + \Delta t^2  \int W^{n+1}_e (\vec{\tilde{b}}-\tilde{A}_s) \vec{G}_W\rd{x}  + \underline{\Delta t  \int W^{n+1}_e \mu A_s(I+\mu  A)^{-1} \tilde{A} \vec{F}_W\rd{x} }+ O(\Delta t^7).
\end{split}\end{equation}

Compared to the previous section and similar estimates for \eqref{energy_e_stst}, we need to treat the underlined terms differently (other terms only contribute  $O(\Delta t^7)$ or terms which can be absorbed). We need a technical lemma which utilizes the condition \eqref{cond_A0}.

\begin{lemma}\label{lem_A0}
\eqref{cond_A0} implies that $(A_s(I+\mu  A)^{-1})_2 = 0$ and $(\vec{\tilde{b}}(I+\mu  A)^{-1})_2 = 0$ for any $\mu>0$.
\end{lemma}

\begin{proof}
We follow the notation in the proof of Lemma \ref{lem_mu}. Notice that under the assumption \eqref{cond_A0}, $D^{-1}L$ has the second column equal to zero, and thus the same is true for $(D^{-1}L)^k$. Therefore \eqref{muA0} shows that $((I+\mu  A)^{-1})_{i,2}=0,\,i=1,3,4,\dots,s$. Combined with \eqref{cond_A0}, we get the conclusion.
\end{proof}

{\bf UNDERLINED TERM 1} only gives a contribution of $O(\Delta t^7)$ due to Lemma \ref{lem_A0}. In fact, first recall the definition of $\vec{F}_U$ and $\vec{F}_W$ in \eqref{FUW}. $\vec{\tilde{b}}\partial_x \vec{E}_U$ and $\vec{\tilde{b}}(I+\mu  A)^{-1}\partial_x \vec{E}_W$ are $O(\Delta t^3)$ because the 2nd component of the coefficient vector is zero due to Lemma \ref{lem_A0}, and other components $\vec{E}_U$ or $\vec{E}_W$ are $O(\Delta t^3)$ by Lemma \ref{lem_LTE3}.

{\bf UNDERLINED TERM 2} gives a contribution of $O(\Delta t^5 \frac{\mu^2}{(1+\mu)^2})$ for the same reason, combining with the fact $\mu A_s (I+\mu  A)^{-1}=O(\frac{\mu}{1+\mu})$ from \eqref{lem_mu_2}.

{\bf UNDERLINED TERM 3} gives a contribution of $O(\Delta t^5 \frac{\mu^2}{(1+\mu)^2})$ by $\mu A_s (I+\mu  A)^{-1}=O(\frac{\mu}{1+\mu})$ and the fact that $\vec{E}_U,\vec{E}_W$ are $O(\Delta t^2)$ for all components (by Lemma \ref{lem_LTE3}).

Therefore we finally get
\begin{equation}\begin{split}
& (1-b^2)\|U_{e}^{n+1}\|^2+\|W_{e}^{n+1}\|^2 
\le  (1+C\Delta t)((1-b^2)\|U^{n}_e\|^2+\|W^{n}_e\|^2) + O\Big(\Delta t^7+\Delta t^5\frac{\mu^2}{(1+\mu)^2}\Big), \\
\end{split}\end{equation}
i.e.,
\begin{equation}\label{final_est1_st}\begin{split}
& (1-b^2)\|U_{e}^n\|^2+\|W_{e}^n\|^2 \le C(T)\Big(\Delta t^6+\Delta t^4\frac{\mu^2}{(1+\mu)^2}\Big),\quad \forall n\Delta t \le T,
\end{split}\end{equation}
which implies second order uniform accuracy, and also implies the desired third order accuracy if $\varepsilon=O(1)$. Thus we may assume $\varepsilon\le 1$ in the rest of this proof, and then $\Delta t^4\frac{\mu^2}{(1+\mu)^2}$ is always the worst term above since $\Delta t\le c$ is assumed. 

\subsection{Improve to third order}

We then improve the error estimate \eqref{final_est1_st} to uniform third order for $W$. We start by revisiting \eqref{sch_vec_stst1b}. Lemma \ref{lem_v} gives
\begin{equation}
\mu\int \vec{W}_{e**}^\top M A \vec{W}_{e**}\rd{x} \ge c\mu \|W_{e**}^{(s)}\|^2,
\end{equation}
which contributes a good term $-c\mu \|W_{e**}^{(s)}\|^2$. We estimate the terms in the first line of \eqref{sch_vec_stst1b} as
\begin{equation}\begin{split}
\Delta t\left| \int W^{(i)}_{e**}\partial_x U^{(j)}_{e**}\rd{x}\right| \le & c\frac{\mu}{1+\mu}\|W^{(i)}_{e**}\|^2 + C\Delta t^2\frac{1+\mu}{\mu}\|\partial_x U^{(j)}_{e**}\|^2 
\le  c\frac{\mu}{1+\mu} \|W^{(i)}_{e**}\|^2 + C\Delta t^6\frac{\mu}{1+\mu},\\
\end{split}\end{equation}
by using \eqref{final_est1_st} for $\|\partial_x U^{(j)}_{e**}\|^2$. The first term $c\frac{\mu}{1+\mu} \|W^{(i)}_{e**}\|^2$ can be absorbed by $-c\mu \|\vec{W}_{e**}^{(s)}\|^2$ together with good terms $-c\|\delta \vec{W}_{e**}\|^2$ in \eqref{sch_vec_stst1b}. 

The last term in \eqref{sch_vec_stst1b} can be easily controlled by $C\Delta t \|W^n_e\|^2 + C\Delta t \|W^{(i)}_{e**}-W^n_e\|^2 + O(\Delta t^7)$.  Therefore we get
\begin{equation}\label{We_energy_st}\begin{split}
&\|W_{e**}^{(s)}\|^2 \le (1+C\Delta t)\|W^{n}_e\|^2 + C\Delta t^6\frac{\mu}{1+\mu} - c  \|\delta \vec{W}_{e**}\|^2 - c\mu \|W^{(s)}_{e**}\|^2.
\end{split}\end{equation}

Adding with  \eqref{energy_e_n01ststb}, we get
\begin{equation}\label{We_energy1_st}\begin{split}
\|W_{e}^{n+1}\|^2 = & (1+C\Delta t)\|W^{n}_{e}\|^2 - c\|W_{e}^{n+1}-W^{(s)}_{e**}\|^2 - c  \|\delta \vec{W}_{e**}\|^2 - c\mu \|W^{(s)}_{e**}\|^2\\
& -\Delta t \int W^{n+1}_e (\vec{\tilde{b}}-\tilde{A}_s) ((1-b^2)\partial_x \vec{U}_{e**} - b\partial_x \vec{W}_{e**})\rd{x}   \\
& -\underline{\Delta t  \int W^{n+1}_e \vec{\tilde{b}} \vec{F}_W\rd{x}}  - \underline{ \int W^{n+1}_e \mu A_s (I+\mu  A)^{-1} \vec{E}_W\rd{x}} +  \int W^{n+1}_e E^{n+1}_W\rd{x} \\
& + \Delta t^2  \int W^{n+1}_e (\vec{\tilde{b}}-\tilde{A}_s) \vec{G}_W\rd{x}  + \underline{\Delta t  \int W^{n+1}_e \mu A_s(I+\mu  A)^{-1} \tilde{A} \vec{F}_W\rd{x}} + C\Delta t^6\frac{\mu}{1+\mu}.\\
\end{split}\end{equation}
Notice that we gain a good term $-c\frac{\mu}{1+\mu} \|W^{n+1}_{e}\|^2$ out of the good terms $- c\|W_{e}^{n+1}-W^{(s)}_{e**}\| - c\mu \|W^{(s)}_{e**}\|^2$. This helps us improve {\bf UNDERLINED TERM 2} estimate as
\begin{equation}
\left|\int  W_{e}^{n+1} \mu A_s (I+\mu A)^{-1} \vec{E}_W\rd{x}\right| \le c\frac{\mu}{1+\mu} \|W^{n+1}_{e}\|^2 + C\frac{1+\mu}{\mu}\Delta t^6\cdot \frac{\mu^2}{(1+\mu)^2},
\end{equation}
and similarly for {\bf UNDERLINED TERM 3}. By estimating other terms as in the previous subsection (which gives $O(\Delta t^7)$), we get
\begin{equation}\begin{split}
& \|W_{e}^{n+1}\|^2 \le (1+C\Delta t)\|W^{n}_{e}\|^2 + \frac{\mu}{1+\mu}(C\Delta t^6 - c \|W^{n+1}_{e}\|^2).
\end{split}\end{equation}
Then a bootstrap argument gives
\begin{equation}\label{final_est2_st}\begin{split}
\|W_{e}^{n}\|^2 \le C(T)\Delta t^6,
\end{split}\end{equation}
i.e., uniform third order accuracy of $W$. 


Finally we improve the error estimate \eqref{final_est1_st} to uniform third order for $U$. In \eqref{sch_vec_stst1a}, now we may estimate as
\begin{equation}
\left|\Delta t\int U^{(i)}_{e**}\partial_x W^{(j)}_{e**}\right| \le C\Delta t \|U^{(i)}_{e**}\|^2 + C\Delta t \|\partial_x W^{(j)}_{e**}\|^2\le C\Delta t \|U^n_e\|^2 +C\Delta t \|U^{(i)}_{e**}-U^n_e\|^2 + C\Delta t^7,
\end{equation}
using \eqref{final_est2_st} for $\partial_x W^{(j)}_{e**}$. By estimating other terms in the same way as before, we obtain
\begin{equation}\begin{split}
& \|U_{e**}^{(s)}\|^2 \le (1+C\Delta t)\|U^{n}_e\|^2 -  c \|\delta \vec{U}_{e**}\|^2 + C\Delta t^7.  \\
\end{split}\end{equation}
The same treatment can be applied to the term $\Delta t\int U^{n+1}_{e}\partial_x W^{(j)}_{e**}\rd{x}$ in \eqref{energy_e_n01ststa}. By estimating other terms in the same way as before and adding it with the previous estimate, 
we obtain
\begin{equation}\begin{split}
& \|U_{e}^{n+1}\|^2 \le (1+C\Delta t)\|U^{n}_e\|^2  + C\Delta t^7.  \\
\end{split}\end{equation}
This gives the uniform third order accuracy of $U$ by the Gronwall inequality, and finishes the proof of Theorem \ref{thm_3rd}.

\section{Numerical verification}
\label{sec:num}

In this section we numerically verify the accuracy of some IMEX-RK schemes applied to the linear hyperbolic relaxation system (\ref{eq2}), where we assume $b=0.6$, $x\in[0,2\pi]$ with periodic boundary condition, and consider various values of $\eps$ from $1$ to $10^{-7}$.

We apply the Fourier-Galerkin spectral method in $x$ with a fixed $N=40$. The initial condition is taken as
\begin{equation}
U_N^0= \mathcal{P}_N e^{\sin(x)},\quad V_N^0= \mathcal{P}_N be^{\sin(x)}.
\end{equation}
To avoid the initial layer, we calculate the exact solution at $T_0=1$ and use it as the initial data for the IMEX-RK schemes. 

We consider the three IMEX-RK schemes listed in Appendix~\ref{app:IMEX-RK}: ARS(2,2,2), ARS(4,4,3), BHR(5,5,3)* with $\Delta t = \frac{8}{N^2}\cdot 2^{-k},\,k=1,2,3,4,5$ and final time $T=2$. The numerical error is computed as the left hand side of (\ref{error}). The results are shown in Figure~\ref{fig1}. One can observe that ARS(2,2,2) and BHR(5,5,3)* achieve uniform second/third order accuracy respectively, verifying the conclusions of Theorems \ref{thm_2nd} and \ref{thm_3rd}. On the other hand, ARS(4,4,3), although being third order when $\eps=O(1)$ and $\eps\rightarrow 0$, suffers from order reduction in the intermediate regime. Indeed, it achieves uniform second order accuracy, which also agrees with our analysis since ARS(4,4,3) only satisfies the conditions in Theorem \ref{thm_2nd} but not in Theorem \ref{thm_3rd}.

\begin{figure}[htp!]
\begin{center}
	\includegraphics[width=0.49\textwidth]{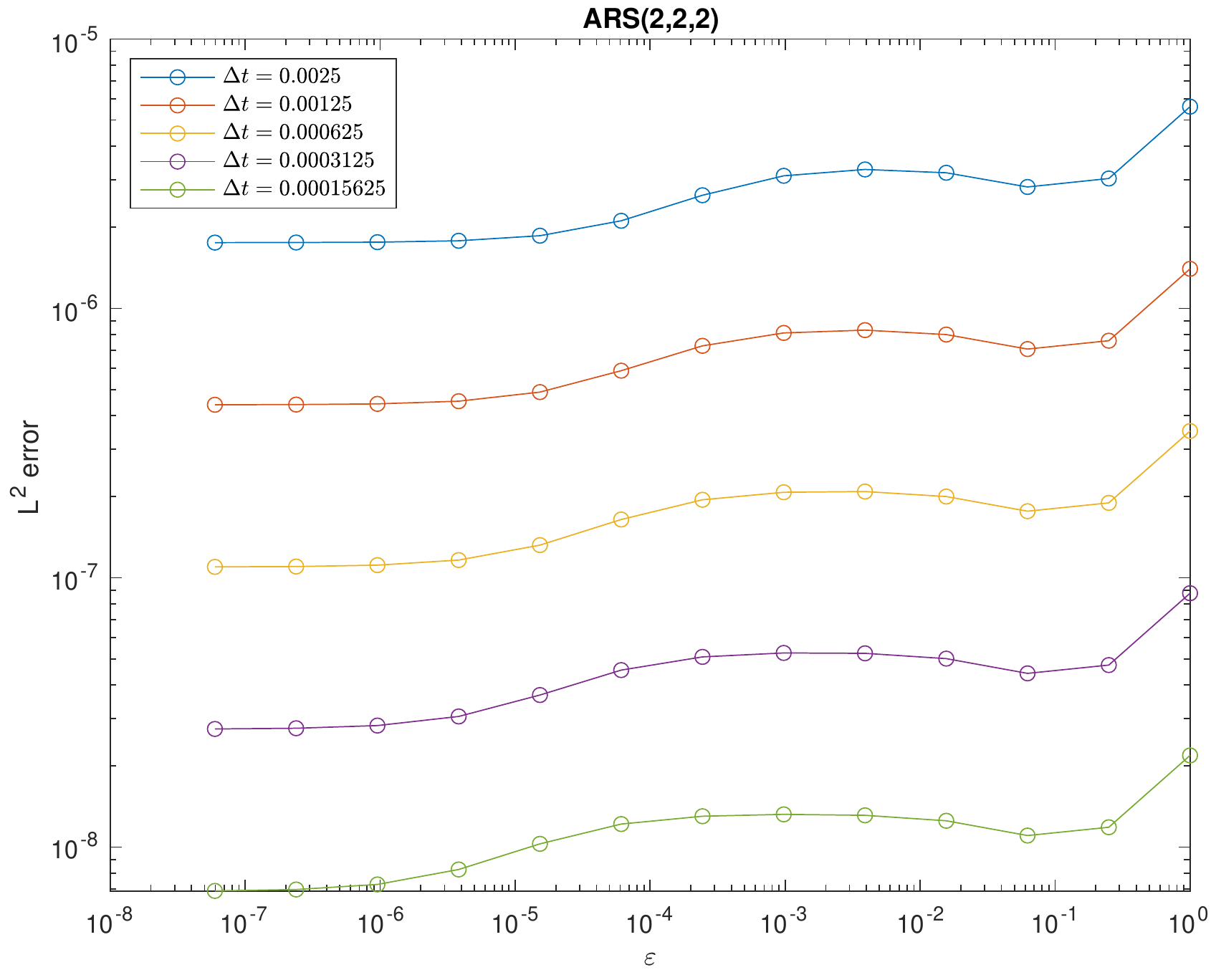}
	\includegraphics[width=0.49\textwidth]{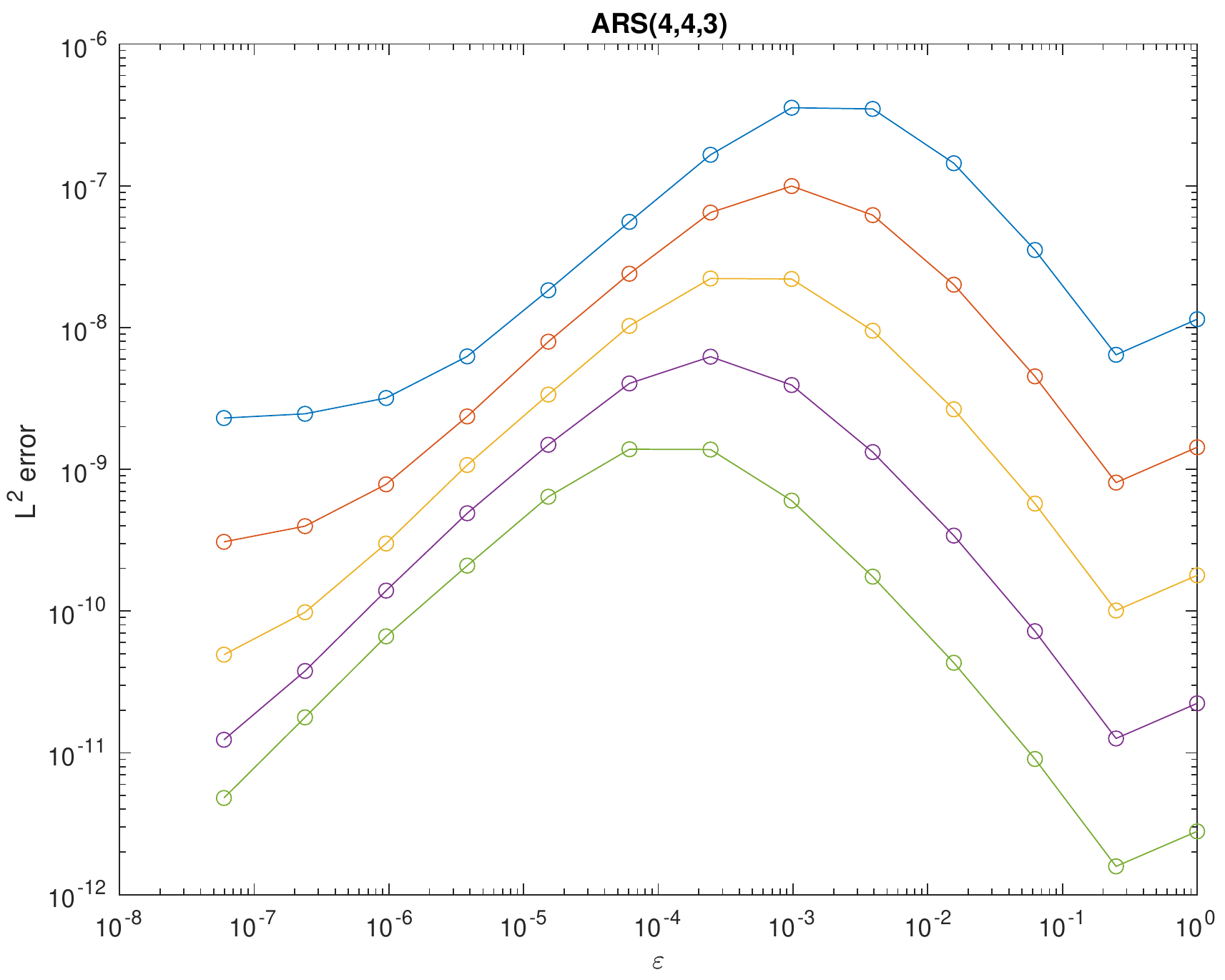}
	\includegraphics[width=0.49\textwidth]{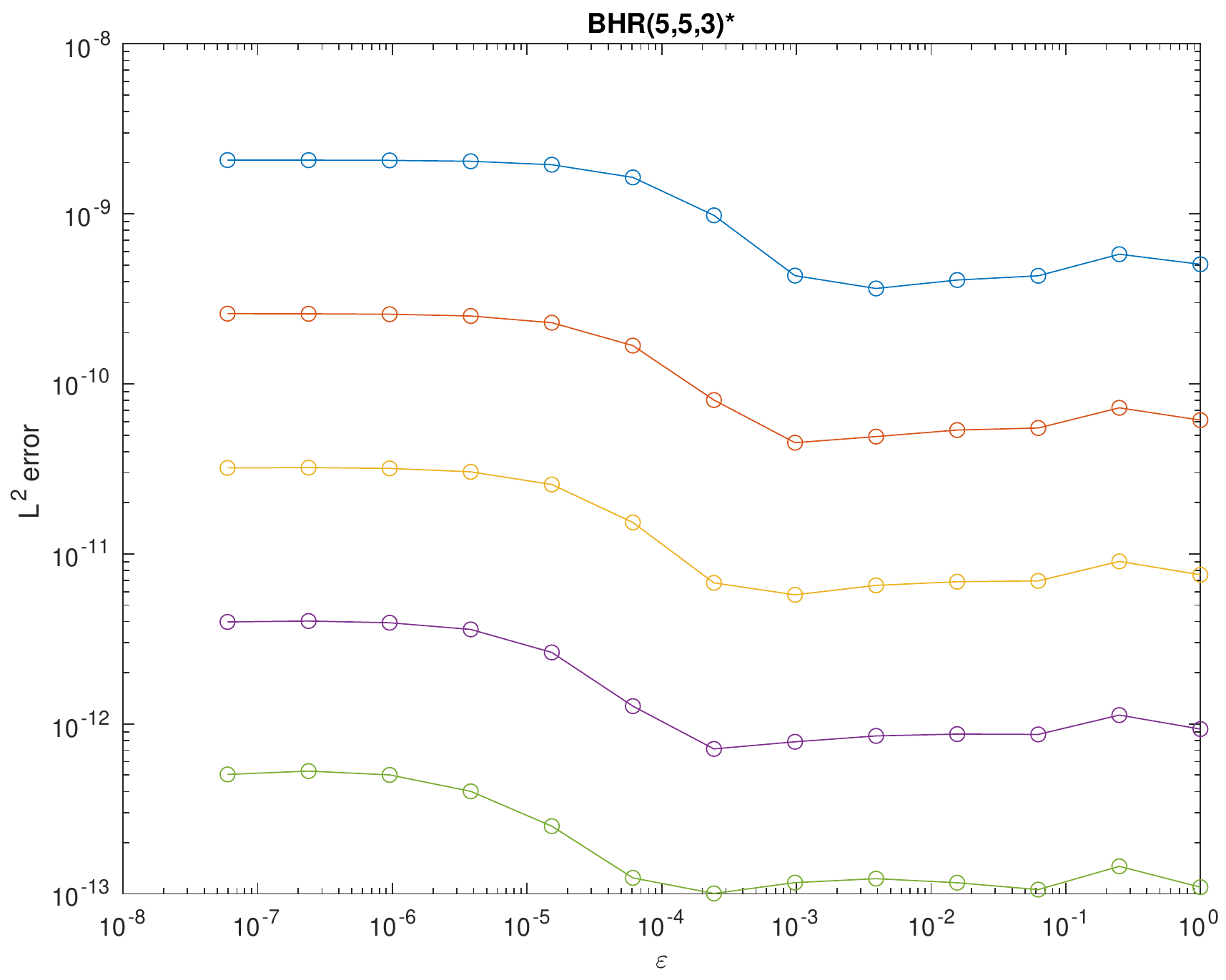}
	\includegraphics[width=0.49\textwidth]{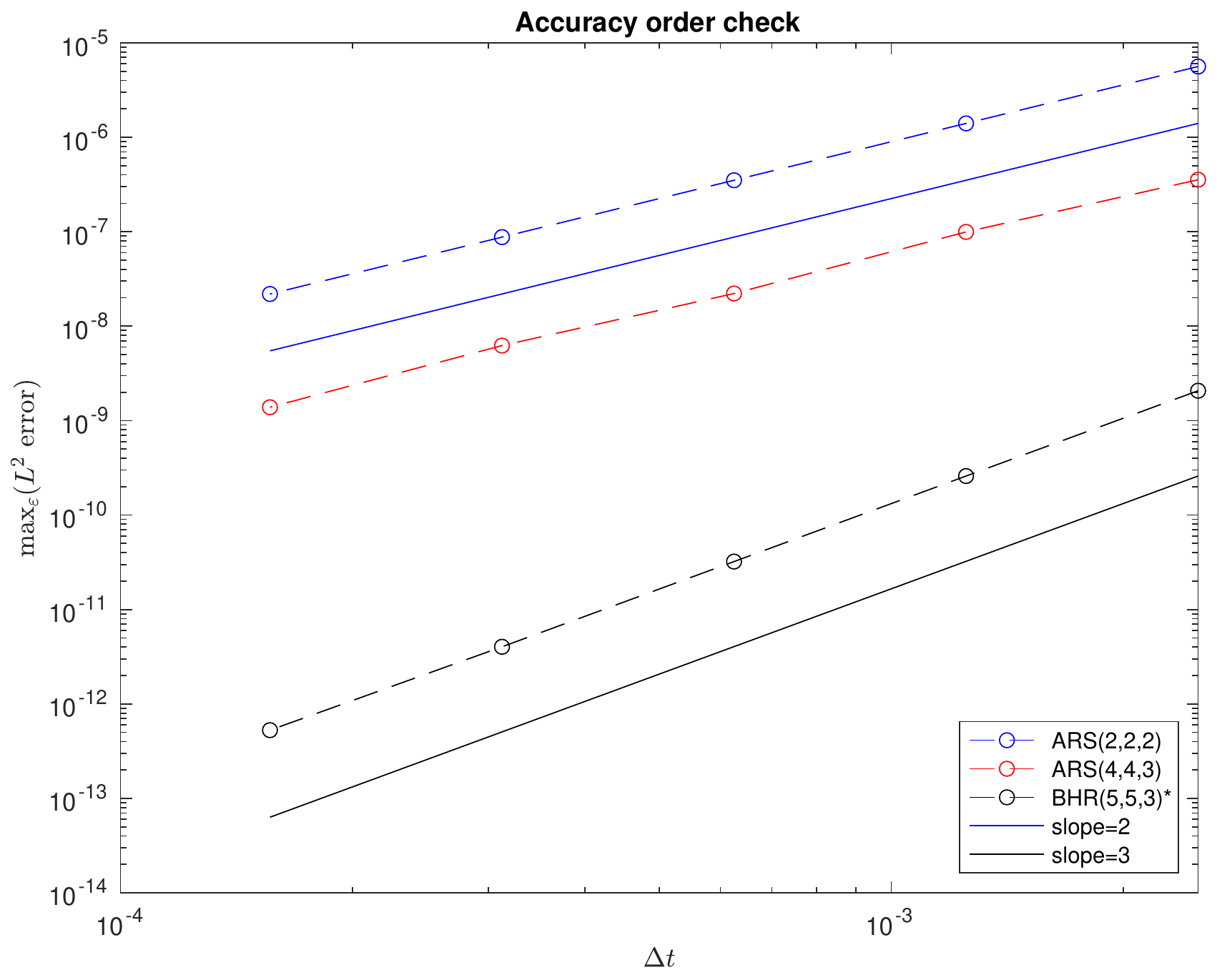}
	\caption{$L^2$ error of the numerical solution to linear hyperbolic relaxation system (\ref{eq2}) computed by IMEX-RK schemes. Top left: ARS(2,2,2); top right: ARS(4,4,3); bottom left: BHR(5,5,3)*. In each of these figures, horizontal axis is $\eps$ ranging from $10^{-7}$ to $1$, and different curves represent different values of $\Delta t$, as shown in the top left figure. Bottom right figure is obtained as follows: for each scheme, take the maximal $L^2$ error among all values of $\eps$ for a fixed $\Delta t$.}
	\label{fig1}
\end{center}	
\end{figure}

\section{Appendix}

\subsection{Necessary conditions for the matrix $M$}
\label{app:M}

To find a matrix $M$ satisfying {\bf (M1)} and {\bf (M2)}, we start from the following observations.
\begin{proposition}\label{prop_posdef}
For any $s\times s$ matrix $M$, we have
\begin{itemize}
\item The matrix $MA + (MA)^\top$ cannot be strictly positive-definite. Furthermore, 
\begin{itemize}
\item {\bf (M1)} implies $\vec{v}^\top M = (*,0,\cdots,0)$, where $\vec{v}$ is a generator of the one-dimensional null space of $A$. Here $*$ denotes an arbitrary real number.
\item If $\vec{v}^\top M = (*,0,\cdots,0)$, then $0$ is an eigenvalue of $MA + (MA)^\top$.
\end{itemize}
\item The matrix $M_* + M_*^\top$ cannot be strictly positive-definite. Furthermore, 
\begin{itemize}
\item {\bf (M2)} implies $\vec{e}^\top M = (*,0,\cdots,0,2)$. 
\item If $\vec{e}^\top M = (*,0,\cdots,0,2)$, then $0$ is an eigenvalue of $M_* + M_*^\top$.
\end{itemize}
\end{itemize}
\end{proposition}

This means one has to construct $M$ very carefully, by requiring the necessary conditions above within the construction. This proposition also allows one to verify {\bf (M1)} and {\bf (M2)} numerically. If $\vec{v}^\top M = (*,0,\cdots,0)$ is satisfied and the eigenvalues of $MA + (MA)^\top$, calculated numerically, has one close to zero and all others positive and away from zero, then this gives a rigorous justification of {\bf (M1)}. The same reasoning works for {\bf (M2)}.

\begin{proof}
We first treat the matrix $MA + (MA)^\top$. As noticed in the proof of Lemma \ref{lem_v}, we have $\vec{v}^\top (MA+(MA)^\top) \vec{v}=0$, which implies that $MA + (MA)^\top$ cannot be strictly positive-definite.

If we assume {\bf (M1)}, then we may apply Lemma \ref{lem_P} with $P=MA+(MA)^\top$ to see that $(MA + (MA)^\top) \vec{v} = \vec{0}$. Therefore
\begin{equation*}
 \vec{0} = (MA + (MA)^\top) \vec{v} = A^\top M^\top \vec{v}.
\end{equation*}
This implies that $M^\top \vec{v}$ is in the null space $A^\top$, which is spanned by $(1,0,\cdots,0)^\top$. Therefore $\vec{v}^\top M = (*,0,\cdots,0)$.

If we assume $\vec{v}^\top M = (*,0,\cdots,0)$, then $M^\top \vec{v}$ is in the null space $A^\top$, and thus
\begin{equation*}
(MA + (MA)^\top) \vec{v} = A^\top M^\top \vec{v} = \vec{0}.
\end{equation*}
Therefore $0$ is an eigenvalue of $MA + (MA)^\top$.

Then we treat the matrix $M_* + M_*^\top$ in a similar way. First, since $\vec{e}^\top (M_* + M_*^\top) \vec{e}=0$, we see that $M_* + M_*^\top$ cannot be strictly positive-definite. Similarly as before, {\bf (M2)} implies $(M_* + M_*^\top) \vec{e}=\vec{0}$, from which one easily deduces $\vec{e}^\top M = (*,0,\cdots,0,2)$. Next,  if $\vec{e}^\top M = (*,0,\cdots,0,2)$, then
\begin{equation*}
(M_* + M_*^\top) \vec{e} = (2,0,\cdots,0,-2)^\top + \begin{pmatrix}
0 & -1 & -1 & \cdots & -1 \\
0 & 1 & 0 & \cdots & 0 \\
0 & 0 & 1 & \cdots & 0 \\
\cdots & \cdots & \cdots & \cdots & \cdots \\
0 & 0 & 0 & \cdots & 1 \\
\end{pmatrix}M^\top \vec{e}=\vec{0}.
\end{equation*}
Therefore $0$ is an eigenvalue of $M_* + M_*^\top$.
\end{proof}


\subsection{Some IMEX-RK schemes and their $M$ matrices}
\label{app:IMEX-RK}

We give some examples of type CK IMEX-RK schemes (by their double Butcher tableau) and their corresponding $M$ matrices satisfying {\bf (M1)} and {\bf (M2)}.

ARS(2,2,2) \cite{ARS97}:
\begin{center}
\begin{tabular}{l|lll}
0 & 0 & 0 & 0 \\
$\gamma$ & $\gamma$ & 0 & 0 \\
1 & $\delta$ & $1-\delta$ & 0 \\
\hline
& $\delta$ & $1-\delta$ & 0
\end{tabular}
\hspace{2cm}
\begin{tabular}{l|lll}
0 & 0 & 0 & 0 \\
$\gamma$ & 0 & $\gamma$  & 0 \\
1 & 0 & $1-\gamma$ & $\gamma$ \\
\hline
&0 & $1-\gamma$ & $\gamma$
\end{tabular}
\end{center}
where $\gamma = 1-\frac{\sqrt{2}}{2}\approx 0.292893218813452$, $\delta = 1-\frac{1}{2\gamma} \approx -0.707106781186548$. The matrix $M$ and vector $\vec{v}$ are given by
\begin{equation*}
M=\begin{pmatrix}
0 & 0 & 0 \\
0 & 2 & 0 \\
0 & -2 & 2 \\
\end{pmatrix}, \quad
\vec{v}=\begin{pmatrix}
1\\
0\\
0
\end{pmatrix}.
\end{equation*}

ARS(4,4,3) \cite{ARS97}:
\begin{center}
\begin{tabular}{l|lllll}
0 & 0 & 0 & 0 & 0 & 0\\
$1/2$ &    $1/2$ & 0 & 0 & 0 & 0\\
$2/3$ &    $11/18$ & $1/18$ & 0 & 0 & 0\\
$1/2$ &    $5/6$ & $-5/6$ & $1/2$ & 0 & 0\\
$1$ &    $1/4$ & $7/4$ & $3/4$ & $-7/4$ & 0\\
\hline
& $1/4$ & $7/4$ & $3/4$ & $-7/4$ & 0\\
    \end{tabular}
\hspace{2cm}
\begin{tabular}{l|lllll}
0 & 0 & 0 & 0 & 0 & 0\\
$1/2$ &     0 & $1/2$ & 0 & 0 & 0\\
$2/3$ &     0 & $1/6$ & $1/2$ & 0 & 0\\
$1/2$ &     0 & $-1/2$ & $1/2$ & $1/2$ & 0\\
$1$ &    0 & $3/2$ & $-3/2$ & $1/2$ & $1/2$\\
    \hline
&0 & $3/2$ & $-3/2$ & $1/2$ & $1/2$\\
\end{tabular}
\end{center}
The matrix $M$ and vector $\vec{v}$ are given by
\begin{equation*}
M=\begin{pmatrix}
0   &  0  &   0 &    0  &   0\\
     0 &   20  & -18  &   0  &   0\\
     0 &  -20  &  20  &   0  &   0\\
     0  &   0 &   -2  &   2   &  0\\
     0  &   0  &   0 &   -2  &   2\\
     \end{pmatrix},
\quad
\vec{v}=\begin{pmatrix}
1\\
0\\
0\\
0\\
0
\end{pmatrix}.     
\end{equation*}

BHR(5,5,3)* (a variant of the BHR(5,5,3) scheme in \cite{BR09}):
\begin{equation*}
\begin{tabular}{c|ccccc}
$0$ & 0  & 0 & 0 & 0 & 0   \\
$2\gamma $ & $2\gamma$  & 0  &  0 & 0 & 0  \\
$2\gamma$ & $\gamma$ & $\gamma$  & 0 & 0 & 0  \\
$c_4$  & $c_4-\frac{c_4^2}{4\gamma}$ & 0  & $\frac{c_4^2}{4\gamma}$ & 0 & 0 \\
$1$ & $1+b_3-\tilde{a}_{53}-\tilde{a}_{54}$  & $-b_3$  & $\tilde{a}_{53}$ & $\tilde{a}_{54}$ & 0  \\
\hline
 & $1-b_3-b_4-\gamma$ &  0  & $b_3$ & $b_4$ &  $\gamma$  
\end{tabular}
\quad
\begin{tabular}{c|ccccc}
$0$ & 0  & 0 & 0 & 0 & 0   \\
$2\gamma$ & $\gamma$  & $\gamma$  &  0 & 0 & 0  \\
$2\gamma$ & $\gamma$ & 0  & $\gamma$ & 0 & 0  \\
$c_4$  & $\frac{3c_4}{2}-\frac{c_4^2}{4\gamma}-\gamma$ & 0  & $\frac{c_4^2}{4\gamma}-\frac{c_4}{2}$ & $\gamma$ & 0 \\
$1$ & $1-b_3-b_4-\gamma$  & 0  & $b_3$ & $b_4$ & $\gamma$  \\
\hline
 & $1-b_3-b_4-\gamma$  & 0  & $b_3$ & $b_4$ & $\gamma$  
\end{tabular}
\end{equation*}
where $\gamma\approx 0.435866521508460$ is the middle root of the polynomial $6\gamma^3-18\gamma^2+9\gamma-1$. $c_4$ is a free parameter. The coefficients $b_3$, $b_4$ are determined by
\begin{equation*}
\begin{pmatrix}
2\gamma & c_4 \\
4\gamma^2 & c_4^2 \\
\end{pmatrix}
\begin{pmatrix}
b_3 \\
b_4 \\
\end{pmatrix}=
\begin{pmatrix}
1/2-\gamma \\
1/3-\gamma \\
\end{pmatrix},
\end{equation*}
and $\tilde{a}_{53},\tilde{a}_{54}$ are determined by
\begin{equation*}
\begin{pmatrix}
2\gamma & c_4 \\
4\gamma^2 & c_4^2 \\
\end{pmatrix}
\begin{pmatrix}
\tilde{a}_{53} \\
\tilde{a}_{54} \\
\end{pmatrix}=
\begin{pmatrix}
1/2+2b_3\gamma \\
1/(12\gamma)-b_4c_4^2 \\
\end{pmatrix}.
\end{equation*}

In particular, if we choose $c_4=1.5$, the above tableau reads (here the last digit may be inaccurate due to round-off error):
{\tiny
\begin{center}
\begin{tabular}{l|lllll}
0&0&0&0&0&0\\
0.871733043016919&0.871733043016919&0&0&0&0\\
0.871733043016919&0.435866521508460&0.435866521508460&0&0&0\\
1.5&0.209467297343041&0&1.290532702656959&0&0\\
1&0.317724380220406&-0.362863385578740&1.195970114894582&-0.150831109536248&0\\
\hline
&0.369394442791758&0&0.362863385578740&-0.168124349878957&0.435866521508460\\
\end{tabular}
\begin{tabular}{l|lllll}
0&0&0&0&0&0\\
0.871733043016919&0.435866521508460&0.435866521508460&0&0&0\\
0.871733043016919&0.435866521508460&0&0.435866521508460&0&0\\
1.5&0.523600775834581&0&0.540532702656959&0.435866521508460&0\\
1&0.369394442791758&0&0.362863385578740&-0.168124349878957&0.435866521508460\\
\hline
&0.369394442791758&0&0.362863385578740&-0.168124349878957&0.435866521508460\\
\end{tabular}
\end{center}}
\noindent The matrix $M$ and vector $\vec{v}$ are given below. It is verified numerically that $MA + (MA)^\top$ and $M_*+M_*^\top$ are semi-positive-definite, and their second smallest eigenvalue is greater than 0.01.
\begin{equation*}
M=\begin{pmatrix}
0&0.043575411705898&0.114355000407169&-0.063096606048326&0.420443444804810\\
0&0.124868150581963&-0.211334272741718&0.127296242100737&0.269577631500343\\
0&-0.076633818924610&0.322397165551093&-0.179513569782584&0.153723853317174\\
0&0.119931254991190&-0.084746378964125&0.280057507208928&0.073572486565781\\
0&-0.211740998354441&-0.140671514252420&-0.164743573478755&1.082682583811894\\
\end{pmatrix}.
\end{equation*}
\begin{equation*}
\vec{v}=\begin{pmatrix}
1\\
 -1\\
  -1\\
   0.038846587170263\\
 0
\end{pmatrix},
\end{equation*}

\begin{remark}
Note that the BHR(5,5,3)* scheme given above is similar to the first scheme in Appendix 2 of \cite{BR09} but different in the value of $c_4$. In fact, $c_4$ in \cite{BR09} is taken approximately as $2.3402$ in order to minimize the fourth-order error constant. However, we are not able to find a matrix $M$ for this scheme. Instead, we take a different value $c_4=1.5$ as above, for which we can find a matrix $M$ satisfying {\bf (M1)} and {\bf (M2)}.
\end{remark}

\subsection{Formal asymptotic-preserving (AP) property of IMEX-RK schemes}

We devote the final section to a formal AP property of the IMEX-RK schemes considered in this paper.  The purpose is to clarify the difference on AP and uniform accuracy. 

Since the proof is formal, we are able to consider a nonlinear hyperbolic relaxation system (c.f., \eqref{eq2}):
\begin{equation} \label{nhr}
\left\{\begin{split}
& \partial_t u + \partial_x v = 0, \\
& \partial_t v + \partial_x u = \frac{1}{\eps}(f(u)-v),
\end{split}\right.\end{equation}
where $f(u)$ is a smooth, nonlinear function of $u$ with $|f'(u)|<1$. Applying the IMEX-RK scheme to (\ref{nhr}) as in \eqref{sch} yields
\begin{equation}\label{sch_f}\begin{split}
& U^{(i)} = U^n -\Delta t \sum_{j=1}^{i-1} \tilde{a}_{i j} \partial_x V^{(j)}, \quad i=1,\dots,s, \\
& V^{(i)} = V^n -\Delta t \sum_{j=1}^{i-1}\tilde{a}_{i j} \partial_x U^{(j)} +\frac{\Delta t}{\varepsilon}\sum_{j=1}^i a_{i j}(f(U^{(j)})-V^{(j)}), \quad i=1,\dots,s,  \\
& U^{n+1} = U^n -\Delta t \sum_{j=1}^{s} \tilde{b}_j \partial_x V^{(j)},\\
& V^{n+1} = V^n -\Delta t \sum_{j=1}^{s}\tilde{b}_j \partial_x U^{(j)} + \frac{\Delta t}{\varepsilon}\sum_{j=1}^s b_j (f(U^{(j)})-V^{(j)}). \\
\end{split}\end{equation}

We first give the definition of an AP scheme for (\ref{nhr}).
\begin{definition}
The scheme (\ref{sch_f}) is AP means: (\ref{sch_f}) is a $q$-th order accurate scheme for (\ref{nhr}) when $\varepsilon=O(1)$; and when $\varepsilon\rightarrow 0$ and keeping $\Delta t$ fixed, (\ref{sch_f}) reduces to a $q$-th order consistent discretization to the limiting equation $\partial_t u +\partial_x f(u)=0$ of (\ref{nhr}).
\end{definition}
Note that our definition here is stronger than the classical AP schemes \cite{Jin_Rev}: the latter only requires that the limiting scheme is a consistent discretization to the limiting equation. Our definition of AP is also known as asymptotically accurate (AA) in the terminology of \cite{DP13}.

We have the following theorem.
\begin{theorem}\label{thm_APnl}
Consider a $q$-th order IMEX-RK scheme (\ref{sch_f}) of type CK and ISA, subject to consistent initial condition: $v_{in}-f(u_{in})=O(\varepsilon)$. Assume $U^n$, $U^{(i)}$, $V^n$, $V^{(i)}$ and their derivatives are $O(1)$ for all $n$ and $i$. When $\varepsilon\rightarrow 0$ and keeping $\Delta t$ fixed, we have
\begin{itemize}
\item If the scheme is GSA, then $V^n=f(U^n),\,V^{(i)}=f(U^{(i)})$ for all $n$ and $i$. Substituting these into the first and third lines of (\ref{sch_f}), we obtain a $q$-th order (explicit) RK scheme applied to the limiting equation $\partial_t u +\partial_x f(u)=0$. Hence the scheme is AP.
\item If the scheme satisfies condition {\bf (A)} or equivalently $\vec{e}_s^\top \hat{A}^{-1}\vec{a}=0$ (c.f., Remark~\ref{eAa}), then $V^n=f(U^n)+O(\Delta t)$, $V^{(i)}=f(U^{(i)})+O(\Delta t)$. Substituting these into the first and third lines of (\ref{sch_f}), we obtain a $q$-th order (explicit) RK scheme applied to the limiting equation $\partial_t u +\partial_x f(u)=0$ with extra error terms of $(\Delta t^2)$. This means the limiting scheme is consistent and at least first order accurate.
\begin{itemize}
\item If in addition $\sum_{j=1}^s(\tilde{b}_j-\tilde{a}_{sj})=0$, then $V^n=f(U^n)+O(\Delta t^2)$, $V^{(i)}=f(U^{(i)})+O(\Delta t^2)$. Using the same argument as above, this means if the original scheme is second order then the limiting scheme also maintains second order, hence it is AP.
\item If in addition $\sum_{j=1}^{s}(\tilde{b}_j-\tilde{a}_{s j})\tilde{c}_j=0$, then $V^n=f(U^n)+O(\Delta t^3)$, $V^{(i)}=f(U^{(i)})+O(\Delta t^3)$. Using the same argument as above, this means if the original scheme is third order then the limiting scheme also maintains third order, hence it is AP.
\end{itemize}
\end{itemize}
\end{theorem}
The GSA result above was already obtained in \cite{DP13}. Here we include its proof for completeness. The ARS(2,2,2) and ARS(4,4,3) schemes are GSA, hence they are AP. The second bullet above is new and identifies a class of schemes that are also AP. In fact, the BHR(5,5,3)* scheme belongs to this category and satisfies all conditions, hence it is AP. 

\begin{proof}
First of all, as $\varepsilon\rightarrow 0$, the second line of \eqref{sch_f} implies $\sum_{j=1}^i a_{i j}(f(U^{(j)})-V^{(j)})=0$. Since the scheme is CK, using the notation in (\ref{CK}) and that $\hat{A}$ is invertible, this can be written as
\begin{equation}\label{dxVfU}
\hat{\vec{V}}-f(\hat{\vec{U}}) =-\hat{A}^{-1}\vec{a}(V^n-f(U^n)),
\end{equation}
where $\hat{\vec{V}}:=(V^{(2)},\dots,V^{(s)})^\top,\,\hat{\vec{U}}:=(U^{(2)},\dots,U^{(s)})^\top$.
Given the consistent initial condition, we see that within the first time step $V^{(i)}=f(U^{(i)})$, $i=2,\dots,s$.

If the scheme is GSA, we have $U^{n+1}=U^{(s)},V^{n+1}=V^{(s)}$. Thus induction yields $V^n=f(U^n)$ and $V^{(i)}=f(U^{(i)})$ for all $n\geq 1$ and associated $i$ stage. This finishes the proof of bullet 1.

We now turn to bullet 2. Note that (\ref{dxVfU}) and $\vec{e}_s^\top \hat{A}^{-1}\vec{a}=0$ implies $V^{(s)}=f(U^{(s)})$. Using that the scheme is ISA, we can write
\begin{equation}\label{sch_f1}
U^{n+1} = U^{(s)} -\Delta t \sum_{j=1}^{s} (\tilde{b}_j-\tilde{a}_{s j}) \partial_x V^{(j)}, \quad V^{n+1} = V^{(s)} -\Delta t \sum_{j=1}^{s} (\tilde{b}_j-\tilde{a}_{s j}) \partial_x U^{(j)}. 
\end{equation}
Then we have
\begin{equation*}
U^{n+1} - U^{(s)} = O(\Delta t),\quad  V^{n+1} - V^{(s)}=O(\Delta t).
\end{equation*}
Hence
\begin{equation} \label{vV}
V^{n+1}-f(U^{n+1})=V^{n+1}-V^{(s)}+f(U^{(s)})-f(U^{n+1})=O(\Delta t), \quad n\geq 0.
\end{equation}
Using again (\ref{dxVfU}), we have for the associated $i$ stage
\begin{equation} \label{V1}
V^{(i)}-f(U^{(i)})=O(\Delta t).
\end{equation}

Now we assume $\sum_{j=1}^s(\tilde{b}_j-\tilde{a}_{sj})=0$. The first equation of \eqref{sch_f1} gives
\begin{equation}\label{AP_est3}
U^{n+1} - U^{(s)} =  -\Delta t \sum_{j=1}^{s} (\tilde{b}_j-\tilde{a}_{s j}) \partial_x (V^{(j)}-f(U^{(j)})) - \Delta t \sum_{j=1}^{s} (\tilde{b}_j-\tilde{a}_{s j}) \partial_x f(U^{(j)}),
\end{equation}
where the first term is $O(\Delta t^2)$ due to \eqref{V1}. To estimate the second term, we have $U^{(j)}-U^n=O(\Delta t)$ from (\ref{sch_f}), then
$$
\sum_{j=1}^{s} (\tilde{b}_j-\tilde{a}_{s j}) \partial_x f(U^{(j)})=\sum_{j=1}^{s} (\tilde{b}_j-\tilde{a}_{s j})  \partial_x (f(U^{(j)})-f(U^n))=O(\Delta t).
$$
Together we have $U^{n+1}-U^{(s)}=O(\Delta t^2)$. Similarly the second equation of \eqref{sch_f1} gives
\begin{equation*}
V^{n+1} - V^{(s)} =  -\Delta t \sum_{j=1}^{s} (\tilde{b}_j-\tilde{a}_{s j}) \partial_x (U^{(j)}-U^n)=O(\Delta t^2).
\end{equation*}
The same argument as in (\ref{vV})-(\ref{V1}) yields $V^{n+1}-f(U^{n+1})=O(\Delta t^2)$ for $n\geq 0$, and for the associated stage
\begin{equation} \label{V2}
V^{(i)}-f(U^{(i)})=O(\Delta t^2).
\end{equation}

Now we further assume $\sum_{j=1}^{s}(\tilde{b}_j-\tilde{a}_{s j})\tilde{c}_j=0$. We will still use (\ref{AP_est3}) to do the estimate. Due to (\ref{V2}), the first term on the right hand side of (\ref{AP_est3}) is $O(\Delta t^3)$. To estimate the second term, we have from (\ref{sch_f})
\begin{equation*}
U^{(j)} = U^n -\Delta t \sum_{k=1}^{j-1} \tilde{a}_{jk} \partial_x V^{(k)}=U^n -\Delta t \sum_{k=1}^{j-1} \tilde{a}_{jk} \partial_x (V^{(k)}-V^n)-\Delta t \tilde{c}_j \partial_xV^n=U^n -\Delta t \tilde{c}_j \partial_xV^n+O(\Delta t^2),
\end{equation*}
where we used $V^{(j)}-V^n=V^{(j)}-f(U^{(j)})+f(U^{(j)})-f(U^n)+f(U^n)-V^n=O(\Delta t)$.
Hence 
$$
f(U^{(j)})=f(U^n)-\Delta t \tilde{c}_jf'(U^n)\partial_xV^n+O(\Delta t^2),
$$
then
\begin{equation*}
\sum_{j=1}^{s} (\tilde{b}_j-\tilde{a}_{s j}) \partial_x f(U^{(j)})=\sum_{j=1}^{s} (\tilde{b}_j-\tilde{a}_{s j}) \partial_x f(U^n)-\Delta t \sum_{j=1}^{s} (\tilde{b}_j-\tilde{a}_{s j}) \tilde{c}_j \partial_x(f'(U^n)\partial_xV^n)+O(\Delta t^2)=O(\Delta t^2).
\end{equation*}
Together we have $U^{n+1}-U^{(s)}=O(\Delta t^3)$. Similarly the second equation of \eqref{sch_f1} gives
\begin{equation*} 
V^{n+1} - V^{(s)} =-\Delta t \sum_{j=1}^{s} (\tilde{b}_j-\tilde{a}_{s j}) \partial_x (U^n -\Delta t \tilde{c}_j \partial_xV^n)+O(\Delta t^3)=O(\Delta t^3).
\end{equation*}
The same argument as in (\ref{vV})-(\ref{V1}) yields $V^{n+1}-f(U^{n+1})=O(\Delta t^3)$ for $n\geq 0$, and $V^{(i)}-f(U^{(i)})=O(\Delta t^3)$ for the associated stage.
\end{proof}

\section*{Acknowledgement}
The work of J. Hu was partially supported by NSF DMS-2153208, AFOSR FA9550-21-1-0358, and DOE DE-SC0023164. The work of R. Shu was supported by the Advanced Grant Nonlocal-CPD (Nonlocal PDEs for Complex Particle Dynamics: Phase Transitions, Patterns and Synchronization) of the European Research Council Executive Agency (ERC) under the European Union's Horizon 2020 research and innovation programme (grant agreement No. 883363). Both authors would like to thank the Isaac Newton Institute for Mathematical Sciences, Cambridge, for support and hospitality during the program `Frontiers in kinetic theory - KineCon 2022' where part of work on this paper was undertaken.

\bibliographystyle{plain}
\bibliography{hu_bibtex}

\begin{thebibliography}{10}

\bibitem{ADP20}
G.~Albi, G.~Dimarco, and L.~Pareschi.
\newblock Implicit-explicit multistep methods for hyperbolic systems with
  multiscale relaxation.
\newblock {\em SIAM J. Sci. Comput.}, 42:A2402--A2435, 2020.

\bibitem{ARS97}
U.~Ascher, S.~Ruuth, and R.~Spiteri.
\newblock Implicit-explicit {R}unge-{K}utta methods for time-dependent partial
  differential equations.
\newblock {\em Appl. Numer. Math.}, 25:151--167, 1997.

\bibitem{ARW95}
U.~Ascher, S.~Ruuth, and B.~Wetton.
\newblock Implicit-explicit methods for time-dependent partial differential
  equations.
\newblock {\em SIAM J. Numer. Anal.}, 32:797--823, 1995.

\bibitem{Boscarino07}
S.~Boscarino.
\newblock Error analysis of {IMEX Runge-Kutta} methods derived from
  differential-algebraic systems.
\newblock {\em SIAM J. Numer. Anal.}, 45:1600--1621, 2007.

\bibitem{Boscarino09}
S.~Boscarino.
\newblock {On an accurate third order implicit-explicit Runge-Kutta method for
  stiff problems}.
\newblock {\em Appl. Numer. Math.}, 59:1515--1528, 2009.

\bibitem{BPR13}
S.~Boscarino, L.~Pareschi, and G.~Russo.
\newblock Implicit-explicit {R}unge-{K}utta schemes for hyperbolic systems and
  kinetic equations in the diffusion limit.
\newblock {\em SIAM J. Sci. Comput.}, 35:A22--A51, 2013.

\bibitem{BR09}
S.~Boscarino and G.~Russo.
\newblock {On a class of uniformly accurate IMEX Runge-Kutta schemes and
  applications to hyperbolic systems with relaxation}.
\newblock {\em SIAM J. Sci. Comput.}, 31:1926--1945, 2009.

\bibitem{CLL94}
G.-Q. Chen, C.~D. Levermore, and T.-P. Liu.
\newblock Hyperbolic conservation laws with stiff relaxation terms and entropy.
\newblock {\em Commun. Pure Appl. Math.}, XLVII:787--830, 1994.

\bibitem{DP13}
G.~Dimarco and L.~Pareschi.
\newblock Asymptotic preserving implicit-explicit {R}unge-{K}utta methods for
  nonlinear kinetic equations.
\newblock {\em SIAM J. Numer. Anal.}, 51:1064--1087, 2013.

\bibitem{DP17}
G.~Dimarco and L.~Pareschi.
\newblock Implicit-explicit linear multistep methods for stiff kinetic
  equations.
\newblock {\em SIAM J. Numer. Anal.}, 55:664--690, 2017.

\bibitem{FS16}
G.~Fu and C.-W. Shu.
\newblock Analysis of an embedded discontinuous {G}alerkin method with
  implicit-explicit time-marching for convection-diffusion problems.
\newblock {\em Int. J. Numer. Anal. Model.}, 1:1, 2016.

\bibitem{HS21}
J.~Hu and R.~Shu.
\newblock {On the uniform accuracy of implicit-explicit backward
  differentiation formulas (IMEX-BDF) for stiff hyperbolic relaxation systems
  and kinetic equations}.
\newblock {\em Math. Comp.}, 90:641--670, 2021.

\bibitem{HZ17}
J.~Hu and X.~Zhang.
\newblock {On a class of implicit-explicit Runge Kutta schemes for stiff
  kinetic equations preserving the Navier-Stokes limit}.
\newblock {\em J. Sci. Comput.}, 73:797--818, 2017.

\bibitem{Jin99}
S.~Jin.
\newblock Efficient asymptotic-preserving ({AP}) schemes for some multiscale
  kinetic equations.
\newblock {\em SIAM J. Sci. Comput.}, 21:441--454, 1999.

\bibitem{Jin_Rev}
S.~Jin.
\newblock Asymptotic preserving ({AP}) schemes for multiscale kinetic and
  hyperbolic equations: a review.
\newblock {\em Riv. Mat. Univ. Parma}, 3:177--216, 2012.

\bibitem{Jin22}
S.~Jin.
\newblock Asymptotic-preserving schemes for multiscale physical problems.
\newblock {\em Acta Numer.}, pages 415--489, 2022.

\bibitem{KC03}
C.~Kennedy and M.~Carpenter.
\newblock Additive {R}unge-{K}utta schemes for convection-diffusion-reaction
  equations.
\newblock {\em Appl. Numer. Math.}, 44:139--181, 2003.

\bibitem{PR05}
L.~Pareschi and G.~Russo.
\newblock Implicit-{E}xplicit {R}unge-{K}utta methods and applications to
  hyperbolic systems with relaxation.
\newblock {\em J. Sci. Comput.}, 25:129--155, 2005.

\end{thebibliography}

\end{document}